\documentclass[final]{siamltex}
\usepackage{amssymb, amsmath}
\usepackage{float,epsfig}
\usepackage{tikz}

\newtheorem{remark}[theorem]{ Remark}
\newtheorem{exam}[theorem]{\bf Example}

\newcommand{\ba}{\begin{array}}
\newcommand{\ea}{\end{array}}
\newcommand{\be}{\begin{equation}}
\newcommand{\ee}{\end{equation}}
\newcommand{\beano}{\begin{eqnarray*}}
\newcommand{\eeano}{\end{eqnarray*}}
\newcommand{\von}{\vskip 1ex}
\newcommand{\vone}{\vskip 2ex}
\newcommand{\vtwo}{\vskip 4ex}
\newcommand{\dm}[1]{ {\displaystyle{#1} } }

\def \R{{\mathbb R}}
\def \C{{\mathbb C}}

\def \cond{{\mathbb{cond}}}

\def \rank{\mathrm{rank}}

\def \d{\mathcal{D}}

\def \sep{\mathrm{sep}}

\def\bmatrix#1{\left[ \begin{matrix} #1 \end{matrix} \right]}

\newcommand{\inp}[2]{\langle {#1} ,\,{#2} \rangle}

\def \dist{\mathrm{dist}}
\def \span{\mathrm{span}}

\def \cond{\mathrm{cond}}

\def \rar{\rightarrow}

\def \lam{\lambda}
\def \sig{\sigma}
\def \Lam{\Lambda}
\def \ep{\epsilon}

\def \diag{\mathrm{diag}}

\def \d{\mathrm{d}}
\def \tr{\mathrm{Tr}}

\def \cnn{\C^{n\times n}}

\def \eig{\mathrm{eig}}

\def \vec{\mathrm{vec}}

\title{ A brief overview of spectral perturbation Theory }

\author{Rafikul Alam \thanks{Department of Mathematics, Indian Institute of Technology Guwahati, Guwahati - 781039, India ({\tt rafik@iitg.ac.in, rafikul68@gmail.com }) Fax: +91-361-2690762/2582649.} }

\begin{document}

\maketitle

\begin{abstract} The aim of this article is to present a brief overview of spectral perturbation theory for matrices,  bounded linear operators and holomorphic operator-valued functions. We focus on bounds for perturbed eigenvalues, eigenvectors and invariant subspaces and provide simplified proofs of some well known results. We present a comprehensive perturbation analysis of invariant subspaces of  matrices.  For bounded linear operators we discuss, among other things, the effect of analytic perturbation on the discrete eigenvalues and spectral projections. We also briefly discuss analytic spectral perturbation theory for holomorphic operator-valued functions.  \end{abstract}

\section{Introduction} Eigenvalue problems arise in many applications in science and engineering. Consider the system of ordinary differential equations \be \label{ode} \frac{dx}{dt} = Ax + f \ee where $A \in \C^{n\times n}$ is a matrix and $x : \R \longrightarrow \C^n$ is differentiable and $ f : \R \longrightarrow \C^n$ is a given function.   Seeking a solution of the form $ x(t) := e^{\lam t} v,$ where $ v \in \C^n$ is nonzero, of the homogeneous equation yields the eigenvalue problem $$ Av = \lam v.$$
If all the eigenvalues of $A$ have negative real part then the system (\ref{ode}) is asymptotically stable. Often in practice the matrix $A$ is not known exactly but only an approximation $A+E$ of $A$ is known due to various reasons such as modeling and approximation errors. Even when $A$ is known exactly, due to rounding errors,  computed eigenvalues of $A$ will at best be exact eigenvalues of $A+E$  for some matrix $E$ such that $\|E\|$ is small. For example, the {\sc matlab} command   {\tt  [U, D] = eig(A)} provides a diagonal matrix $D \in \C^{n\times n}$ whose diagonal entries are computed eigenvalues and a matrix $U \in \C^{n\times n}$ whose columns are computed eigenvectors satisfying  $(A+ E)U = UD$  for some matrix  $E$ such that $\|E\| $ is bounded by a constant multiple of the unit roundoff. Since eigenvalues and eigenvectors of $A$ provide a solution of the system (\ref{ode}), it is important to understand the effect of the perturbation $A+E$ on the eigenvalues and eigenvectors of $A.$ For instance, if the matrix $A$ is stable (i.e, eigenvalues of $A$  lie in the open left half complex plane), will the matrix $A+E$ be stable? If $A$ has distinct eigenvalues, will $A+E$ have distinct eigenvalues?   

\vone 
Next, consider the second order system of ordinary differential equations \be\label{qep}  \ddot{y}(t) + A(\tau) \dot{y}(t) +B(\tau) y(t) =0, \ee where $A(\tau)$ and $B(\tau)$ are  matrices in $\C^{n\times n}$ which depend on a parameter $ \tau \in \C$ holomorphically  and $ y : \R \longrightarrow \C^n$ is a smooth function. Here $\dot{y}$ denotes the derivative of $y$ and $\ddot{y}$ denotes the second derivative of $y.$ Now, for a fixed $\tau \in \C,$ seeking a solution of (\ref{qep}) of the form $ y(t) := e^{\lam t} v,$ where $ v \in \C^n$ is nonzero, yields the quadratic eigenvalue problem (QEP)
\be  \label{quad} (\lam^2 I + \lam A(\tau) + B(\tau))v = 0. \ee  Setting $ v_1 := \lam v$ and $ v_2 := v$, the QEP can be rewritten as a one parameter matrix eigenvalue problem  \be\label{gep}   \bmatrix{ -A(\tau) & -B(\tau) \\  I & 0} \bmatrix{  v_1 \\ v_2} = \lam \bmatrix{  v_1 \\ v_2}   \Longrightarrow  \mathcal{A}(\tau)\mathbf{v} = \lam \mathbf{v}.\ee Since the  eigenvalues and eigenvectors of $\mathcal{A}(\tau)$ are required for solutions of (\ref{qep}), it is important to analyze the eigenvalues and eigenvectors of $\mathcal{A}(\tau)$ when $ \tau $ varies in $\C.$

\vone 
Finally, given  continuous  functions $ K : [a, b]\times [a, b] \longrightarrow \C$ and $ g :[a, b] \longrightarrow \C$, consider the Fredholm integral equation of the second kind \be \label{int} \int^b_a K(s, t) f(t) dt  - \lam f(s) = g(s) \text{ for } s \in [a, b],\ee where $\lam \in \C$ and $f: [a, b] \longrightarrow \C$ is continuous and a solution of (\ref{int}). A solution of the homogeneous equation yields the eigenvalue problem $$ \int^b_a K(s, t) f(t) dt  =  \lam f(s) \Longrightarrow  Tf = \lam f.$$
The infinite dimensional eigenvalue problem $Tf = \lam f$ is approximated by a  finite dimensional problem $T_n f_n = \lam_n f_n,$ where $T_n$ is a sequence of finite rank linear operators representing, for instance, projection, Galerkin or Nystr\"{o}m  method~\cite{limbook2}. Note that $ T $ is a bounded linear (compact) operator on $C[a, b]$. However, $T_n$ often does not converge to $T$ in the norm. It is therefore imperative to investigate the mode of convergence $ T_n \longrightarrow T$ so that the  eigenvalues, eigenvectors  and spectral subspaces of $T$ can be approximate by those of $T_n$ as $ n \rightarrow \infty.$ We will not discuss this issue here and refer to \cite{alamjma, alammc,mario, chatelin, osborn} and the references therein for more on this topic.

\vone 
A central problem in spectral perturbation theory is to analyze variations of eigenvalues, eigenvectors and invariant subspaces of a square matrix or a bounded linear operator $A$ when $A$ undergoes a perturbation of the form $A+E.$ Let $\sig(A)$ be the spectrum of $A$. Then an immediate question that arises is this: How  are $\sig(A+E)$ and $\sig(A)$ related?  Can $ \sig(A)$ expand or shrink suddenly  when $A$ is perturbed to  $A+E$  for a small $\|E\|$?  In other words, is the set-valued map $ A \longmapsto \sig(A)$ continuous?  For instance, if $A$ is an $n\times n$  matrix then a result due to Elsner~\cite{bhatiabook, stewbook} states that 
\be \label{elsbd} d_H(\sig(A), \sig(A+E)) \leq (\|A\|_2 + \|A+E\|_2)^{1-1/n} \| E\|_2^{1/n},\ee where  $ d_H(\sig(A), \sig(A+E))$  is the Housedroff distance. Although, the bound in (\ref{elsbd}) shows that the map  $ A \longmapsto \sig(A)$ is continuous, an unsavory fact about the bound is the  appearance of the $n$-th root $\|E\|_2^{1/n}.$  To see the implication of this bound, suppose that $\|E\|_2 = 10^{-16}$ and $ n= 16.$ Then $ d_H(\sig(A), \sig(A+E)) = \mathcal{O}(10^{-1}).$ This means that an error of magnitude $10^{-16}$ in the matrix $A$ is magnified $10^{15}$ times in the bound for the spectrum of $A+E.$ 
\vone 

On the other hand, if $A$ is a bounded linear operator on an infinite dimensional Banach space then the continuity of the map $ A \longmapsto \sig(A)$ is not guaranteed. However, $\sig(A)$ cannot expand suddenly when $A$ is perturbed to $A+E$ for a small $\|E\|.$ In other words, the map  $ A \longmapsto \sig(A)$ is upper semicontinuous~(see, \cite{limbook2, kato}), that is, for any open set $U$ containing $\sig(A)$ there exists a $\delta >0$ such that $$ \|E\| < \delta \Longrightarrow  \sig(A+E) \subset U.$$ This leaves the question about sudden shrinkage of $\sig(A)$ when $A$ is perturbed to $A+E$. In other words, is the map  $ A \longmapsto \sig(A)$  lower semicontinuous, that is, if $U$ is an open set  such that $ \sig(A)\cap U \neq \emptyset$ then does there exist $\delta >0$ such that $$ \|E\| < \delta \Longrightarrow \sig(A+E)\cap U  \neq \emptyset?$$ It is well known~\cite{kato, limbook2} that the map  $ A \longmapsto \sig(A)$  is not lower semicontinuous. We illustrate this fact by an example; see~\cite{kato} for details.

Consider the Hilbert space $\ell^2(\mathbb{Z}) :=\left\{ x: \mathbb{Z} \longrightarrow \C \;  | \; \sum_{j \in \mathbb{Z}} |x(j)|^2 < \infty\right\}$ and  the orthonormal basis  $\{ e_n : n \in \mathbb{Z}\},$ where $e_n(j) = \delta_{nj}$ and $\delta_{nj}$ is the Kronecker delta.  Now consider the left shift operator $ S : \ell^2(\mathbb{Z}) \longrightarrow \ell^2(\mathbb{Z})$ given  by 
$$S e_0 =0 \; \text{ and } \; S e_n = e_{n-1} \; \text{ for } \; n \neq 0.$$   Then  $  \sig(S) =  \{ \lam \in \C : |\lam| \leq 1\} = :\mathbb{D}.$ In fact,  if $|\lam| <1 $ then for $ u := \sum^\infty_{n=0} \lam^n e_n,$ we have  $ S u = \lam u$ showing that $\lam$  is an eigenvalue of  $S$   when   $|\lam| < 1.$

Now define $ E : \ell^2(\mathbb{Z}) \longrightarrow \ell^2(\mathbb{Z})$ by $ E e_0 = e_{-1}$ and $ Ee_n = 0$ for $n \neq 0.$ Consider the perturbed operator $A(\tau) := S + \tau E$ for $ \tau \in \C.$  Then $ A(\tau) e_0 = \tau e_{-1}$ and $ A(\tau) e_n = e_{n-1}$ for $ n\neq 0$ is a left shift operator and the spectral radius $r_{\sig}((A(\tau))) =1.$   Hence  $\sig(A(\tau)) \subset \mathbb{D}.$

Note that if  $ \tau \neq 0$ then $A(\tau)$ is invertible and $A(\tau)^{-1}$ is a right shift operator with spectral radius $r_{\sig}((A(\tau))^{-1}) =1.$  Hence  $\sig(A(\tau)^{-1}) \subset \mathbb{D}$   for $\tau \neq 0.$  By the spectral mapping theorem   
$ \sig(A(\tau)) \subset  \{ \lam \in \C : |\lam| =1\} $ for  $\tau \neq 0.$ This shows that the set-valued map $ \tau \longmapsto \sig(A(\tau))$ is NOT  lower semicontinuous. Notice  that  the spectrum $\sig(A(\tau))$ at $\tau:=0$ shrinks suddenly when the  perturbation is switched on. This example also illustrates that the spectral perturbation theory for bounded linear operators cannot be expected to be a routine generalization of  the spectral perturbation theory for  matrices. Hence we treat the two cases separately.

%
\vone

Spectral perturbation theory for matrices and linear operators is a classical subject and has been studied extensively over the years; see~\cite{tofna, alamela,alambora, alambora2, alambora3, alambora4, bhatiabook, baum, chu, stewbook, limbook2, kato, hornbook, sun1, sun4,W1, W2, W3, W4} and references therein. The main objective of this article is to provide a brief overview of spectral perturbation  theory for matrices and bounded linear operators. We mainly focus on perturbation bounds for (discrete) eigenvalues, eigenvectors and invariant subspaces of matrices and  bounded linear operators. In particular, we present a comprehensive perturbation analysis of invariant subspaces of matrices and provide simplified proofs of some well known results.  The bound in (\ref{elsbd}) shows that for any $\lam \in \sig(A)$ there exists $ \lam_E \in \sig(A+E)$ such that $$|\lam - \lam_E| \leq  \kappa \|E\|_2^{1/n},$$ where $\kappa$ is a constant. This bound illustrates  the drawback of  {\em ``one bound fits all eigenvalues"} strategy. It is a fact that each eigenvalue of $A$ behaves differently when $A$ is perturbed to $A+E.$ Hence it is imperative to focus on individual eigenvalues or a small group of eigenvalues and derive perturbation bounds for the eigenvalues, their corresponding eigenvectors and the associated spectral subspace  when $A$ is perturbed to $A+E.$  With this motivation, we review three kinds of bounds, namely, asymptotic bounds, local bounds and global bounds for perturbed eigenvalues. 

\vone 
An asymptotic bound for a (discrete) eigenvalue $\lam$ is a bound of the form 
$$ |\lam - \lam_E|^\nu  \leq \kappa(\lam, A) \|E\| + \mathcal{O}(\|E\|^2)$$ for sufficiently small $\|E\|,$ where $\nu$ is the ascent of $\lam$ and $\kappa(\lam, A)$ is a constant. In contrast, a local bound  puts a restriction on $\|E\|$ to derive a non-asymptotic bound
$$ \|E\|< \delta \Longrightarrow |\lam - \lam_E|^\nu  \leq \kappa(\lam, \delta, A) \|E\|.$$ On the other hand, a global bound puts no restriction on $\|E\|$ to derive a non-asymptotic bound of the form  
$$ |\lam - \lam_E|^\nu  \leq \kappa(\lam, A) \|E\| \;\;\text{ for all } E. $$  We mention that when $A$ has additional properties such as when $A$ is selfadjoint or when $A$ is a structured matrix (e.g., Hamiltonian, skew-Hamiltonian, Toeplitz, Hankel, unitary and symplectic - to name only a few) then structured perturbation theory provides specialized results under structure-preserving perturbation. However, we will not discuss structured perturbation theory in this article and refer to \cite{alamvolker, mehl1, mehl2, mehl3, moro1, moro2, rump, bhatiabook, bibhas1, bibhas3, bibhas2} and  the references therein for more on this topic. \vone

The rest of this article is organized as follows. In section~2,  we present preliminary results.  In section~3, we analyze perturbation of invariant subspaces of matrices. We derive perturbation bounds for eigenvalues of matrices in section~4. In section~5, we consider  parameter dependent eigenvalue problem and discuss analyticity  of eigenvalues and spectral projections. Finally, in section~6, we consider spectral perturbation theory for holomorphic operator-valued functions.


\vone 

\section{Preliminaries} 
Let $\C^{m\times n}$ denote the set of all $m$-by-$n$ matrices with entries in $\C.$  Let $A
\in \C^{n\times n}$ and $\lam$ be an eigenvalue of $A,$ that is, $\rank(A-\lam I) < n.$ Then
there exist  nonzero vectors $x \in \C^n$ and $ y \in \C^n$ such
that $$Ax = \lam x  \mbox{ and } y^*A = \lam y^*,$$ where $y^*$
denotes the conjugate transpose of $y.$  The vectors $y$ and $x$ are
called left and right eigenvectors of $A$ corresponding to $\lam,$
respectively. We refer to $(\lam, y, x)$ as an eigentriple of $A.$
An eigenvalue $\lam$ is simple if it is a simple root of the characteristic polynomial $p(z) :=\det(zI - A).$  We refer to $(\lam, y, x)$ as a simple  eigentriple of $A$   when $\lam$ is a simple  eigenvalue of $A.$ We denote the spectrum of $A$ by $\eig(A)$ and the resolvent set of $A$ by $\rho(A)$,  that is, $$ \eig(A) := \{ \lam \in \C : \rank(A- \lam I) < n\} \;\; \text{ and } \;\; \rho(A) = \C \setminus \eig(A).$$

 We denote the Kronecker product of two matrices $A$ and $B$ by $ A \otimes B.$  The $2$-norm of a vector $ x \in \C^n$ is given by $\|x\|_2 := \sqrt{|x_1|^2+ \cdots + |x_n|^2}$ and the spectral norm (i.e., the $2$-norm) of  matrix $ A \in \C^{m\times n}$ is given by $$\|A\|_2 := \max\{ \|Ax\|_2 : x \in \C^n, \;  \|x\|_2=1\}.$$ The Frobenius norm of $A$ is given by $ \|A\|_F := \sqrt{\sum^m_{i=1}\sum^n_{j=1} |a_{ij}|^2}.$   We denote the $n\times n$ identity matrix by $I_n$ and also by $I$ when the size is clear from the context. The $j$-th column of $I_n$ is denoted by $e_j.$ Thus $e_1, \ldots, e_n$ are canonical vectors which form a standard basis of $\C^n.$  For $ X, Y \in \C^{m\times n},$ $ \inp{X}{Y} := \tr(Y^*X)$ is the Frobenius inner product on $\C^{m\times n}.$ \von

Let $X$  be a complex Banach space and 	$BL(X)$ denote the Banach space of all bounded linear operators on $X.$  Let  $ A \in BL(X).$  Then $A$ is said to be  invertible if $A^{-1}$ exists and $A^{-1} \in BL(X).$   The norm of $A$ is given by   $$ \|A\| := \sup\{ \|Ax\| : x \in X, \|x\|=1\} = \sup\{ \|Ax\| : x \in X, \|x\|\leq 1\}.$$ The resolvent set $\rho(A)$ and the spectrum $\sig(A)$ of $A$ are given by 
$$ \rho(A) := \{ z \in \C : (A- zI)^{-1} \in BL(X)\} \text{ and } \sig(A) := \C\setminus \rho(A).$$  Obviously we have $ \sig(A) = \eig(A)$ when $A$ is a matrix.  Let $ \mu \in \C.$ If $X$ is finite dimensional then 	
$$ A- \mu I \text{ is injective } \iff A- \mu I \text{ is surjective } \iff A - \mu I \text{ is invertible. } $$
This is no longer true when $X$ is infinite dimensional. For instance, consider  $$ A : C[0, 1] \longrightarrow C[0, 1], \; f(s) \longmapsto \int^s_0 f(t)dt.$$ Obviously, $A$ is injective. Since $(Af)(0) =0$ for all $f \in C[0, 1]$,  $A$ is NOT surjective.    

\vone This is an indication that  $\sig(A)$ is likely a more complicated set when $A$ is an infinite dimensional bounded linear operator than $\eig(A)$ when $A$ is a matrix. For instance, an infinite dimensional bounded linear operator  $A$ may not have an eigenvalue.    We denote the eigenspectrum of $A$ by $\eig(A)$ and is given by $$ \eig(A) := \{ \mu \in \C :  A- \mu I \text{ is not injective }\}.$$ If $\eig(A)$ is nonempty and $ \lam \in \eig(A)$ then there exists a nonzero vector $ v \in X$ such that $Av = \lam v.$ In such a case, $\lam $ is called an eigenvalue of $A$ and $ v$ is called  a corresponding eigenvector. 

\vone 

It is well known~\cite{limbook1, limbook2, kato}  that the spectrum $\sig(A)$ is nonempty and compact, and the resolvent set $\rho(A)$ is open.  The {\em spectral radius} of $A$ is given by   $$r_\sig(A) := \max\{ | \lam| : \lam \in \sig(A)\}.$$ 	
It is also well known~\cite{limbook1} that $r_{\sig}(A) = \inf_n \|A^n\|^{1/n} = \lim_{n\rar \infty} \|A^n\|^{1/n}.$   Let $N(A)$ denote the kernel (null space) of $A$ and $ R(A)$ denote the range space of $A$, that is, 
 $$ N(A) := \{ x \in X : Ax = 0\} \text{ and } R(A) := \{ Ax : x \in X\}.$$  For the special case when $A$ is a matrix, we  write $\span(A)$ to denote the column space of $A$ (span of the columns of $A$) which is the same as the range space $R(A).$  
\vone

An operator $A \in BL(X)$ is said to be {\em Fredholm} if $ R(A)$ is closed and  the quotient space $ X/{R(A)}$ and $N(A)$ are finite dimensional. The index of a Fredholm operator $A$ is defined by $\mathrm{ind}(A) := \mathrm{dim}(N(A)) - \mathrm{dim}(Y/{R(A)})$,  see~\cite{ggk, kato}.

\vone 
  For $z \in \rho(A)$, let $R(z) := (A- zI)^{-1}$ denote the resolvent operator. Then the first resolvent identity $R(z) - R(w) = (z- w) R(z) R(w)$ holds for all $ z , w \in \rho(A).$ 
  
  \vone 
  \begin{lemma}\cite{limbook2} \label{pqrank} Let $ P$ and $ Q$ be  projections in $BL(X).$ If $ r_{\sig}(P-Q) <1 $ then $\rank(P) = \rank(Q). $ Further, the maps  $$ J_P : R(Q) \longrightarrow R(P), x \longmapsto Px, \mbox{ and } J_Q: R(P) \longrightarrow  R(Q), y \longmapsto Qy,$$ are linear isomorphisms.  \end{lemma}

  %

  A consequence of Lemma~\ref{pqrank} is that a projection-valued  function $ t \longmapsto P(t)$ continuous on a connected domain has a constant rank. 
  
  \begin{proposition}\label{contp} Let $D \subset \C$ be connected. Let $ P : D \longrightarrow BL(X)$ be continuous and $P(t)^2 = P(t)$ for all $ t \in D.$ Let $ s \in D.$  Then $\rank(P(t)) = \rank(P(s))$ for $ t \in D.$
  \end{proposition}
  \von
  \begin{proof} 	Since $P(t)$ is continuous on $D,$ for every $t_0 \in D$ there is a $\delta_0>0$ such that $ t \in D$ and  $ |t-t_0| < \delta_0 \Longrightarrow \|P(t) - P(t_0)\| < 1.$ Hence by by Lemma~\ref{pqrank},   $ \rank(P(t)) = \rank(P(t_0)).$
  	Let $ S := \{ t \in D :  \rank(P(t)) = \rank(P(s))\}.$ Then $ S$ is both closed and open in $D.$ Since $D$ connected, we have $ S= D.$  
  \end{proof} 
 
 \vone 
 \noindent{\bf Invariant subspace:} A subspace $Y \subset X$ is said to be an invariant subspace of $A$ if $ AY \subset Y.$  Let $ P\in BL(X)$ be a nontrivial projection, that is, $ P^2 = P$ and $ 0 \neq P \neq I.$ Set $ Y := R(P)$ and $ Z := N(P).$ Then $X$ is decomposed as $$\begin{array}{cc} X = Y\oplus Z &  \text{ with } x = y+z \rightsquigarrow \bmatrix{ y\\ z}, \text{ where }  y \in Y \text{ and } z \in Z \end{array}. $$    
  Then $A$ can be written as a $2$-by-$2$ operator matrix  $$ A = \bmatrix{ A_{11} & A_{12} \\ A_{21} & A_{22}} = \bmatrix{ PAP & PA(I-P) \\ (I-P)AP & (I-P)A(I-P)}. $$ Note the abuse of notation in the last equality! For instance, $PAP $ denotes the restriction $PA_{|R(P)}$  on $R(P)$ and so on.  Then  \beano A Y \subset Y &\Longleftrightarrow& A_{21} = 0\Longleftrightarrow(I-P)AP  = 0 \iff AP = PAP \\   A Z \subset Z &\Longleftrightarrow& A_{12} = 0 \Longleftrightarrow PA(I-P)  = 0 \Longleftrightarrow PA = PAP.\eeano Hence $ AY \subset Y$ and $ AZ \subset Z \iff AP = PA \iff A = \bmatrix{ A_{11} & 0 \\ 0 & A_{22}}.$

 \vone 
 Suppose that  $PA= AP.$ Then $ A Y\subset Y$ and $ AZ \subset Z.$  In such a case, the invariant subspaces   $Y$ and $Z$ are called {\em reducing subspaces} of $A.$   
 Set $ A_Y := PAP = {P AP}_{|Y}$ and $ A_Z := (I-P)A(I-P) = {(I-P) A(I-P)}_{|Z}.$ Then  $X = Y\oplus Z$ and  $$ A = \bmatrix{ A_Y & 0 \\ 0 & A_Z} \text{  and  } \sig(A) = \sig(A_Y) \cup\sig(A_Z).$$ If $ \sig(A_Y) \cap \sig(A_Z) = \emptyset$ then $ Y$ (resp., $Z$ ) is called the  {\em spectral subspace} of $A$ corresponding to $\sig(A_Y)$ (resp., $\sig(A_Z).$) The projection  $P$ is called the {\em spectral projection} of $A$ corresponding to $\sig(A_Y).$ 
 
 \vone
 \begin{remark}  A spectral projection $P$ associated with $A$ decomposes   $ X  = Y \oplus Z$ and $A = A_Y  \oplus A_Z $ as well as the spectrum  $\sig(A) = \sig(A_Y) \cup \sig(A_Z)$ with $\sig(A_Y) \cap \sig(A_Z) = \emptyset,$ where $ Y := R(P)$ and $ Z:= N(P).$ 
 	\end{remark} 
 
 \vone \begin{remark}  Let $ P \in BL(X)$ be a projection and $ \sig_0 \subset \sig(A)$ be compact.  Then $P$ is  the  spectral projection  of $A$ corresponding to $\sig_0$ if and only if $ PA = AP, \sig(A_Y) = \sig_0$ and $ \sig(A_Z) = \sig(A)\setminus \sig_0,$  where $ Y := R(P)$ and $ Z:= N(P).$ 
 \end{remark} 
 \vone

 We always assume that $\Gamma$ is a positively oriented  rectifiable simple closed curve in $\C$. Let $\mathrm{Int}(\Gamma)$  denote the open region enclosed by $\Gamma$ and $\mathrm{Ext}(\Gamma)$ denote the open region outside the closed region enclosed by  $\Gamma$ so that $$ \mathrm{Int}(\Gamma) \cap \mathrm{Ext}(\Gamma) = \emptyset \; \text{and } \; \C = \mathrm{Int}(\Gamma) \cup \Gamma \cup  \mathrm{Ext}(\Gamma).$$

 Let  $ \Gamma \subset \rho(A)$.  Suppose that $ \sig_0 := \sig(A) \cap \mathrm{Int}(\Gamma)\neq \emptyset.  $  Define  $$\dm{ P := \frac{-1}{2 \pi i} \int_{\Gamma} R(z) dz}\;\; \text{ and }\;\; S(z_0) := \frac{1}{2 \pi i} \int_{\Gamma} \frac{R(z)}{z-z_0} dz \;\; \text{ for } z_0 \in \mathrm{Int}(\Gamma). $$ Then $AP = PA$ and $P$ is a projection. Indeed,  choose $\hat \Gamma \subset \rho(A)$ such that $\Gamma \subset \mathrm{Int}(\hat\Gamma)$ and $\hat \Gamma $ can be continuously deformed in $\rho(A)$ to $\Gamma$. Then  
 \beano P^2  &=&  \frac{-1}{2 \pi i} \int_{\Gamma} R(z) dz  \times \frac{-1}{2 \pi i} \int_{\hat \Gamma} R(w) dw =  \frac{1}{(2 \pi i)^2} \int_{\Gamma} \left(  \int_{\hat \Gamma} R(z)R(w) dw  \right) dz \\ 
 &=&\frac{1}{(2 \pi i)^2} \int_{\Gamma} \left(  \int_{\hat \Gamma} \frac{R(z)-R(w)}{z-w} dw  \right) dz = \frac{1}{(2 \pi i)^2} \int_{\Gamma} R(z)\left(  \int_{\hat \Gamma} \frac{dw}{z-w}   \right) dz \\
 &=& \frac{-1}{2 \pi i} \int_{\Gamma} R(z) dz  = P. \eeano The last equality holds 
 since $\dm{\int_{\Gamma} \left( \int_{\hat \Gamma} \frac{ R(w) dw}{z-w}\right)dz = \int_{\hat\Gamma}\left( R(w) \int_{\Gamma} \frac{dz}{z-w}\right) dw = 0.}$ \\
 
  Since $A$ commutes with $R(z)$, it follows that $A$ commutes with $P.$  The operator $S$ is called the {\em reduced resolvent} of $A- z_0I.$   In fact, $P$ is the spectral projection of $A$ corresponding to $\sig_0.$ This follows from the following result. 
 
 \vone 
 \begin{theorem}[spectral decomposition, \cite{limbook2, kato}] \label{spd} Let $A, P$ and $ \sig_0$ be as above.  Set $ Y := R(P)$ and $ Z := N(P).$ Then $ X = Y\oplus Z$ and $ A = \bmatrix{ A_Y & 0 \\ 0 & A_Z}.$  Further $ \sig(A) = \sig(A_Y)\cup \sig(A_Z)$ and $ \sig(A_Y) = \sig(A) \cap \mathrm{Int}(\Gamma) $ and $ \sig(A_Z) = \sig(A) \cap \mathrm{Ext}(\Gamma).$  Thus $P$ is the spectral projection of $A$ corresponding to $\sig_0 = \sig(A_Y).$ 
 
 \end{theorem}
 
 \vone Now suppose that $\sig_0 = \{ \lam\}.$ Then for $z$ close to $\lam$, the following Laurant series expansion holds~\cite{limbook2, kato}
 \be \label{laurent} (A-zI)^{-1} = \sum^{\infty}_{j=0} S^{j+1}(z - \lam)^j - \frac{P}{z- \lam} - \sum^{\infty}_{j=1} \frac{D^j}{(z-\lam)^{j+1}},\ee
 where $ D := (T-\lam I)P$ and $S$ is the reduced resolvent of $A- \lam I.$ \\

 \begin{definition}\cite{limbook2} Let $ \lam \in \sig(A).$ Then $\lam$ is said to be a discrete eigenvalue of $A$ if $\lam$ is an isolated point of $\sig(A)$ and the spectral projection $P_{\lam} $ associated with $A$ and $\lam$ has finite rank. In such a case, $ m(\lam, A) := \rank(P_{\lam})$ is called the algebraic multiplicity of $\lam.$  The set $\sig_d(A)$ of all discrete eigenvalues of $A$ is called the discrete spectrum of $A.$ Let $ \lam \in \sig_d(A).$ If $\lam$ is a pole of order $\nu$ of $ (A-zI)^{-1}$ then $\nu$ is called the ascent of $\lam.$
 \end{definition}
 
 \von 
Let $ \lam \in \sig_d(A)$ and $P$ be the spectral projection of $A$ corresponding to $\lam.$ If $\nu$ is the ascent of $\lam$ then $ R(P) = N((A- \lam I)^\ell),$ see~\cite{limbook2}.

 \von
 \begin{remark} 
 It follows from (\ref{laurent}) that $\lam$ is a discrete eigenvalue $\Longleftrightarrow \lam$ is a pole of $(A- zI)^{-1}.$ Further, $ \nu$ is the ascent of $\lam$ if and only if $ D^{\nu-1} \neq 0$ and $D^{\nu} =0.$ Now if $ \sig_0 := \mathrm{Int}(\Gamma) \cap \sig(A) $ and $m(\sig_0, A) := \rank(P)$ is finite then $\sig_0$ consists of at most  $m(\sig_0, A)$  discrete eigenvalues of $A$ of total algebraic multiplicity $m(\sig_0, A).$
 \end{remark}

 \vone

 \begin{theorem} \cite{limbook2}
 	Suppose that $ \sig_0 := \sig(A)\cap \mathrm{Int}(\Gamma) = \{\lam_1, \ldots, \lam_n\}\subset \sig_d(A).$ Let $ P_j$ denote the spectral projection corresponding to $\lam_j.$ Then  
 	$$ P := - \frac{1}{2\pi i} \int_{\Gamma} R(z) dz = P_1+ \cdots + P_n.$$Further, $ P_iP_j = 0$ for $i\neq j$ and $ R(P) = R(P_1) \oplus \cdots \oplus R(P_n).$ Furthermore,  $$ AP = \sum^n_{j=1} (\lam_jP_j + D_j), \mbox{ where } \; D_j := (A-\lam_j I )P_j.$$ Also  $ \sig(A_{|R(P)}) = \sig_0 $ and $ \sig(A_{|N(P)}) = \sig(A)\setminus \sig_0.$ 
 \end{theorem}
 
 \vone

 \begin{proof}  Use contour integration. Note that $$ AP = AP_1+\cdots +AP_n = \sum^n_{j=1} (\lam_jP_j + D_j).$$\end{proof}

\section{Perturbation of invariant subspaces}
Let $ A \in \C^{n\times n}.$ A subspace $\mathcal{X} \subset \C^n$ is said to be an {\em invariant subspace} of $A$ if  $A \mathcal{X} \subset \mathcal{X}.$ Let $ X \in \C^{n\times m}$ be such that $\rank(X) =m.$  Let  $\span(X)$ denote the span of the columns of $X$.  Then $\mathcal{X} := \span(X)$ is an invariant subspace of $A$ if and only if    $ A X = XM$ for some matrix $M \in \C^{m\times m}.$  Now choose $Z$ such that  $ \bmatrix{ X & Z}\in \C^{n\times n}$ is nonsingular. Then  \be \label{upper} \bmatrix{X & Z}^{-1} A \bmatrix{X & Z}= \bmatrix{ M & C \\ 0 & N} \text{ and } \eig(A) = \eig(M) \cup \eig(N).\ee If $\eig(M)\cap \eig(N) = \emptyset$ then $\span(X)$ is called  the {\em  spectral subspace} (or maximal invariant subspace) of $A$ corresponding to $\eig(M).$

\vone 

\begin{definition} Let $ A \in \C^{n\times n}.$ A matrix pair $(X, M) \in \C^{n\times m}\times \C^{m\times m}$ is said to be a right invariant pair of $A$  if $\rank(X) = m$ and $ AX = XM.$ Further, $(X, M)$ is called a  maximal right invariant pair if $\span(X)$ is a spectral subspace of $A.$ 	Similarly, a matrix pair $(Y, L) \in \C^{n\times m}\times \C^{m\times m}$ is said to be a left invariant pair of $A$ if $\rank(Y) = m$ and $ Y^*A = L Y^*.$ Further, $(Y, L)$ is called a maximal left invariant pair if $\span(Y)$ is a spectral subspace of $A^*.$ 
\end{definition}

\vone
\noindent
{\bf Question:} Let $ (X, M)$ be a maximal right invariant pair of $A$ and $ E \in \C^{n\times n}.$  Does there exist a maximal right invariant pair $(X_E, M_E) $ of $ A+E$ such that $$ (X_E, M_E) \longrightarrow (X, M)  \text{ as } \|E\|_2 \rar 0?$$  
Set $\mathcal{X}_E := \span(X_E)$ and $ \mathcal{X} := \span(X).$  Does  $\mathcal{X}_E \longrightarrow \mathcal{X}$ as$\|E\|_2 \longrightarrow 0$?

\vone 
 
 Perturbation theory of spectral subspaces of matrices and operators is a classical topic and has been studied extensively, see for instance~\cite{stewart71,stewart73, stewbook, bhatiabook, kato, limbook2, chatelin}. A new bound has been obtained in~\cite{karow} by utilizing pseudospectrum of a matrix.  Inspired by~\cite{karow}, we  present a detailed analysis of perturbation of maximal invariant pairs and provide simplified proofs.   

\vone
Suppose that $ X := \bmatrix{ X_1 & X_2} \in \C^{n\times n}$ is nonsingular with $ X_1 \in \C^{n\times m}$ such that $$ X^{-1} AX = \bmatrix{ A_{11} & 0 \\ 0 & A_{22}} \; \text{ and  }\; \eig(A_{11})\cap \eig(A_{22}) = \emptyset, \text{ where } A_{11} \in \C^{m\times m}.$$  Define $ Y := (X^{-1})^*$ and partition $Y$ conformably as $Y = \bmatrix{ Y_1 & Y_2}$ with $ Y_1 \in \C^{n\times m}.$ Then $ Y^*AX = \diag(A_{11}, \; A_{22}).$ Further, $ (X_1, A_{11}) $ is a right invariant pair of $A$ and $ (Y_1, A_{11})$ is a left invariant pair of $A$. Similarly, $ (X_2, A_{22})$ is a right invariant pair of $A$ and $ (Y_2, A_{22})$ is a left invariant pair of $A.$ Furthermore, $ P_1 := X_1Y_1^*$ is the spectral projection of $A$ corresponding to $ \eig(A_{11})$ and $ P_2 := X_2 Y_2^*$ is the spectral projection of $ A$ corresponding to $ \eig(A_{22}).$ Obviously $ P_1 + P_2 = I$ and $\span(X_1) = R(P_1)$ is the spectral subspace of $A$ corresponding to $\eig(A_{11})$ and $\span(X_2) = R(P_2)$ is the spectral subspace of $A$ corresponding to $ \eig(A_{22}).$  Let $ \Gamma \subset \rho(A) $ be such that $ \eig(A_{11}) \subset \mathrm{Int}(\Gamma)$ and $ \eig(A_{22}) \subset \mathrm{Ext}(\Gamma).$ Then observe that $$ P = \frac{1}{2\pi i} \int_{\Gamma} (zI-A)^{-1} dz = X_1Y_1^*$$

\vone 
{\bf Block diagonalization.}  Let $ \widehat{A}$ denote the block upper triangular matrix in (\ref{upper}). Then $\widehat{A}$  is similar to the block diagonal matrix  $\diag(M, N) \Longleftrightarrow$  there exists $ R \in \C^{m\times (n-m)} $ such that 
$$ \bmatrix{ I_m & R\\ 0 & I_{n-m}} ^{-1}\bmatrix{ M & C \\ 0 & N} \bmatrix{ I_m & R\\ 0 & I_{n-m}} = \bmatrix{ M & 0 \\ 0 & N} \Longleftrightarrow MR- RN = -C.$$ Hence  $ A$ is block diagonalizable with diagonal blocks $M$ and $N$, that is, $A$ is similar to $ \diag(M, N) \Longleftrightarrow R$ is a  solution of the {\em Sylvester equation} $ MR-RN = -C.$ 
Suppose that $R$ exists. Define $$ X_1:= \bmatrix{ I_m \\ 0}, X_2 := \bmatrix{ R\\ I_{n-m}},  Y_1 :=\bmatrix{ I_m \\ -R^*} \text{ and } Y_2 := \bmatrix{ 0 \\ I_{n-m}}.$$ Then $ (X_1, M)$ is a right invariant pair of $\widehat{A}$ and $ (Y_1, M)$ is a left invariant pair of $\widehat{A}.$ Similarly, $(X_2, N)$ is a right invariant pair of $\widehat{A}$ and $(Y_2, N)$ is a left invariant pair of $\widehat{A}$. Further, if $\eig(M)\cap \eig(N) = \emptyset$ then $$ P_1 := X_1Y_1^* = \bmatrix{ I_m & -R\\ 0 & 0} \text{ and } P_2:= X_2Y_2^* = \bmatrix{ 0 & R \\ 0 & I_{n-m}}$$
are the spectral projections of $\widehat{A}$ corresponding to $ \eig(M)$ and $\eig(N)$, respectively. \\

Thus, in view of (\ref{upper}), it follows that $ (X, M)$ and $(XR+Z, N)$ are right invariant pairs of $A$  if and only if $ MR-RN = -C.$ The right invariant pairs of $A$ naturally lead to the Sylvester operator  $$ \mathbf{T} :  Y \longmapsto MY-YN.$$   It is well known~\cite{stewart73,stewbook}  that  $\mathbf{T}$  is invertible if and only if $ \eig(M)\cap\eig(N) = \emptyset$. Indeed, $\mathbf{T}(Y) = 0 \Longleftrightarrow  MY = YN \Longleftrightarrow p(M)Y = Y p(N)$ for any polynomial $p(z) \in \C[z].$  If $  \eig(M)\cap\eig(N) = \emptyset$ then, by considering $p(z) := \det( zI- M)$, it follows that $ Y = 0.$ Hence $\mathbf{T}$ is invertible. On the other hand, if $ \mu \in  \eig(M)\cap\eig(N) $ then there exist nonzero vectors $u$ and $ v$ such that $ Mu = \mu u$ and $ v^* N = \mu v^*$. Now  $ \mathbf{T}( uv^*) = Muv^* - uv^*N = \mu (uv^*- uv^*) = 0 \Longrightarrow \mathbf{T}$ is not invertible. \\

The separation of $\eig(M)$ and $\eig(N)$ ensures invertibility of $\mathbf{T}.$ For perturbation analysis of invariant subspaces, we now define separation of two square matrices~\cite{stewart73, stewbook}. 

\vone
\begin{definition} Let $ A \in \C^{m\times m}$ and $ B\in \C^{n\times n}$. Then the  separation  of  $A$ and $B$ is defined by   $$ \sep(A, B) := \min\{ \|AX-XB\|_2 : X \in \C^{m\times n} \text{ and } \|X\|_2 =1\}.$$ 
	
\end{definition}

\vone  For $ A \in \C^{m\times m}$ and $ B\in \C^{n\times n}$, consider the Sylvester operator $$ \mathbf{T} : \C^{m\times n} \longrightarrow \C^{m\times n}, \; X \longmapsto AX - XB.$$
Observe that  $\sep(A, B) >0 \iff \mathbf{T}$ is invertible $\iff \eig(A)\cap \eig(B) = \emptyset.$ Further, if $\mathbf{T}$ is invertible then $\sep(A, B) =  1/{\|\mathbf{T}^{-1}\|}.$   For the Frobenius norm,   \beano  \sep_F(A, B) &=&  \min\{ \|AX-XB\|_F : \|X\|_F =1\}  \\  &=&  \min\{ \|\vec(AX- XB)\|_2 : \|\vec(X)\|_2 =1\} \\ &= &
\min\{ \| (I_m\otimes A) - B^\top \otimes I_n)\vec(X)\|_2 : \|\vec(X)\|_2 =1\} \\ & =& \sigma_{\min}(I_m\otimes A - B^\top \otimes I_n). \eeano Here $\vec(X)$ denotes the column vector obtained by stacking columns of $X$ starting from the first column at the top and $A\otimes B$ is the Kronecker product of $A$ and $B.$ \vone 

Let $ \dist(S_1, S_2)$  denote the minimum distance between two subsets $S_1$ and $  S_2$  of $\C$ given by $$ \dist(S_1, S_2) := \inf\{ |s_1- s_2| : s_1 \in S_1 \text{ and } s_2 \in S_2\}.$$ 
It is easy to see that  $ \sep(A, B) \leq \dist(\eig(A), \eig(B)).$ However, $\sep(A, B)$ can be much smaller than $\dist(\eig(A), \eig(B)).$  It is well known~\cite{stewbook} that $\sep(A, B)$ as a function of $(A, B)$ is Lipschitz continuous. Indeed, we have the following result which follows from the definition of $\sep(A, B).$ See also \cite{stewart73, stewbook}. 

\vone 

\begin{proposition}\label{lipsep} We have  $ | \sep(A+E, B+F) - \sep(A, B) | \leq \|E\| + \|F\|$ for all $ E \in \C^{m\times m}$ and $ F \in \C^{n\times n}.$
\end{proposition} 
\vone 

{\bf Angle between subspaces.} Let $ x, y \in \C^n$ be unit vectors. Then there exists a unique  $ \theta(x,y) \in [0, \pi/2] $ such that  $ \cos\theta(x, y) = | y^*x|.$ Thus $\theta(x,y)$ is the acute angle between $x$ and $y.$  Note that  $P_y := yy^*$ is an orthogonal projection on $\span(y)$ and that $ \cos\theta(x, y) = \|P_yx\|_2.$  Since   $ x = P_y x+ (I-P_y)x $, we have   $$ 1 = \|x\|_2^2 = \|P_yx\|_2^2 + \|(I- P_y)x\|_2^2 = \cos^2\theta(x, y) + \|(I-P_y)x\|_2^2$$ which shows that $ \sin\theta(x, y) = \|(I-P_y)x\|_2.$ 

\vone 

Let $ X, Y \in \C^{n\times m} $ be  isometry, that is, $X^*X = I_m = Y^*Y.$  Set $ \mathcal{X} := \span(X)$ and $ \mathcal{Y} := \span(Y).$    Then  $P_{\mathcal{X}}  := XX^*$ and $ P_{\mathcal{Y}} := YY^*$ are orthogonal projections on $\mathcal{X}$ and $\mathcal{Y},$ respectively.  Now, for a unit vector $u \in \mathcal{X},$ define the angles $\theta(u, \mathcal{Y}) \in [0, \pi/2]$ and $\theta_{\max}(\mathcal{X}, \mathcal{Y}) \in [0, \pi/2]$ by $$ \sin\theta(u, \mathcal{Y}) :=  \|(I - P_{\mathcal{Y}})u\|_2   \text{ and } \sin\theta_{\max}(\mathcal{X}, \mathcal{Y}) := \max_{\|u\|_2=1}\{ \sin\theta(u, \mathcal{Y}) : u \in \mathcal{X}\}.$$   Since $u \in \span(\mathcal{X})$ is a unit  vector, we have $ u = X v$ for some unit vector $v \in \C^m.$ Hence 
\beano \sin^2\theta(u, \mathcal{Y}) &=& \|(I - P_{\mathcal{Y}})u\|_2^2 = \|(I - YY^*) Xv\|_2^2  = v^* X^*( I- YY^*)Xv \\ & =& 1 - v^*X^*Y Y^*Xv \leq 1 - \lam_{\min}(X^*YY^*X)\eeano and the equality holds when $v$ is a unit eigenvector of $ H := X^*Y Y^*X$ corresponding to the smallest eigenvalue $ \lam_{\min}(H).$ Note that $\sig_{\min}(Y^*X) := \sqrt{\lam_{\min}(H)}$ is the smallest singular value of $Y^*X.$ Now, taking maximum over all unit vector $u \in \mathcal{X},$ we have 
$$\sin\theta_{\max}(\mathcal{X}, \mathcal{Y}) = \sqrt{1 - \sig^2_{\min}(Y^*X)} \Longrightarrow   \cos\theta_{\max}(\mathcal{X}, \mathcal{Y}) = \sig_{\min}(Y^*X).$$

This proves the following result. \vone 

\begin{proposition} \label{prop:angle} Let $ X, Y \in \C^{n\times m} $ be  isometry. Set $ \mathcal{X} := \span(X)$ and $ \mathcal{Y} := \span(Y).$ Then  $P_{\mathcal{X}}  := XX^*$ and $ P_{\mathcal{Y}} := YY^*$ are orthogonal projections on $\mathcal{X}$ and $\mathcal{Y},$ respectively. Define $\theta_{\max}(\mathcal{X}, \mathcal{Y}) \in [0, \pi/2]$ by $$ \sin\theta_{\max}(\mathcal{X}, \mathcal{Y}) := \max_{\|u\|_2=1}\{ \|(I - P_{\mathcal{Y}})u\|_2   : u \in \mathcal{X}\}.$$ Then 
$$\sin\theta_{\max}(\mathcal{X}, \mathcal{Y}) = \sqrt{1 - \sig^2_{\min}(Y^*X)} \text{ and }    \cos\theta_{\max}(\mathcal{X}, \mathcal{Y}) = \sig_{\min}(Y^*X).$$
\end{proposition}

\vone In view of Proposition~\ref{prop:angle}, we now define the largest canonical angle between the subspaces $\mathcal{X}$ and $ \mathcal{Y}$ as follows. 
\vone 
\begin{definition} Let $ X, Y \in \C^{n\times m} $ be  isometry. Set $ \mathcal{X} := \span(X)$ and $ \mathcal{Y} := \span(Y).$	Define $\theta_{\max}(\mathcal{X}, \mathcal{Y}) \in [0, \pi/2]$ by  $ \cos \theta_{\max}(\mathcal{X}, \mathcal{Y}) :=  \sig_{\min}(Y^*X),$ where $\sig_{\min}(A)$ denotes the smallest singular value of a matrix $A.$  Then   $\theta_{\max}(\mathcal{X}, \mathcal{Y}) $ is called the  largest canonical angle between $\mathcal{X}$ and $\mathcal{Y}.$ 
\end{definition}

 \vone 
 \noindent  {\bf Exercise:}  Show that $ \sin\theta_{\max}(\mathcal{X}, \mathcal{Y}) = \|(I- P_{\mathcal{Y}})P_{\mathcal{X}}\|_2 =  \|(I- P_{\mathcal{X}})P_{\mathcal{Y}}\|_2 = \sin \theta_{\max}(\mathcal{Y}, \mathcal{X}).$  Deduce that  $ \sin \theta_{\max}(\mathcal{X}, \mathcal{Y}) =0 \iff \mathcal{X} = \mathcal{Y}.$ Further, show that $$\sin \theta_{\max}(\mathcal{X}, \mathcal{Y}) =  \|P_{\mathcal{X}} - P_{\mathcal{Y}}\|_2.\; \blacksquare$$

 \vone We have the following result which will be used later. \vone 
 
 \begin{proposition}\label{angle}
 	Consider $ \mathcal{X}  := \span\left(\bmatrix{ I_m \\ 0}\right)$ and $ \mathcal{Y} :=  \span\left(\bmatrix{ I_m \\ Z}\right),$  where $ Z \in \C^{(n-m)\times m}.$ Then   $$  \cos \theta_{\max}(\mathcal{X}, \mathcal{Y})   = ( 1+ \|Z\|_2^2)^{-1/2} \text{ and }    \tan\theta_{\max}(\mathcal{X}, \mathcal{Y}) =  \|Z\|_2.$$
 
 \end{proposition} 
 
 \vone 
 \begin{proof}
  Set $  X := \bmatrix{ I_m \\ 0} $ and $  Y := \bmatrix{ I_m \\ Z}( I_m+ Z^*Z)^{-1/2}$.  Then $ Y$ is an isometry and $\mathcal{Y} = \span(Y).$   We have 
 $$\cos \theta_{\max}(\mathcal{X}, \mathcal{Y}) = \sig_{\min}(Y^*X)  = \sig_{\min}(( I_m+ Z^*Z)^{-1/2})  = ( 1+ \|Z\|_2^2)^{-1/2}$$ and $\sin\theta_{\max}(\mathcal{X}, \mathcal{Y}) = \|Z\|_2 /{ ( 1+ \|Z\|_2^2)^{1/2}}.$ 
 Hence $  \tan\theta_{\max}(\mathcal{X}, \mathcal{Y}) =  \|Z\|_2.$
 \end{proof} 
 \vone \von
 
 Pictorially, the result in Proposition~\ref{angle}  can be illustrated as follows.

 $$\mathcal{X} := \span\left( \bmatrix{ I\\ 0}\right)	\text{ and }   \mathcal{Y} := \span\left( \bmatrix{ I\\ Z}\right)$$
 
 \von
 \begin{center}
 	\begin{tikzpicture}
 		
 		\draw [thick, black, ->] (0,-.5) -- (0,2)      
 		node [above, black] {$\bmatrix{0 \\ I}$};              
 		
 		\draw [thick, black, ->] (-1,0) -- (3,0)      
 		node [right, black] {$\bmatrix{I\\ 0}$};              
 		
 		%
 		\draw [draw=blue, thick, ->] (0,0) -- (2.5981, 1.5) 
 		node [right, black]{{\small $\bmatrix{I\\ Z} $}};

 		\draw[->] (1,0) arc (0:10:3cm);  
 		
 		
 		\draw [thick, blue, ->] (0,0) -- (2.5981,0) 
 		node [] at (1.3, .3) {{\small $\;\;\theta_{\max}$}};  
 	\end{tikzpicture} 
 \end{center}

 $$ \text{Figure-1:} \; \cos \theta_{\max}(\mathcal{X}, \mathcal{Y}) =   ( 1+\|Z\|_2^2)^{-1/2} \text{  and  }   \tan\theta_{\max}(\mathcal{X}, \mathcal{Y}) = \|Z\|_2.$$

\vone

 Now, we turn to block upper triangularization of $A$ and the existence of right invariant pairs.  If $ A := \bmatrix{ A_{11} & A_{12} \\ 0 & A_{22}}$ with $ A_{11} \in \C^{m\times m}$ then obviously  $ \left( \bmatrix{ I_m \\ 0}, A_{11}\right)$  is a right invariant pair of $A.$   We now investigate the existence of a right  invariant pair of $ A := \bmatrix{ A_{11} & A_{12} \\ A_{21} & A_{22}}$ of the form $ \left( \bmatrix{ I_m \\ X}, A_{11}+A_{12}X\right)$ \\ 

%
%
%
%
%
%
%

\begin{theorem} \label{lem:ricati}
	Let $ A := \left[ \begin{matrix} A_{11} & A_{12} \\ A_{21} & A_{22}   \end{matrix}\right] \in \C^{n\times n},$ where $A_{11} \in \C^{m\times m}.$  Then there exists $X \in \C^{(n-m)\times m} $ such that  $$ \left[ \begin{matrix} I_m &  0 \\ X & I_{n-m}  \end{matrix}\right]^{-1} \left[ \begin{matrix} A_{11} & A_{12} \\ A_{21} & A_{22}   \end{matrix}\right] \left[ \begin{matrix} I_m &  0 \\ X & I_{n-m}  \end{matrix}\right] =  \left[ \begin{matrix} A_{11}+ A_{12}X  &  A_{12} \\ 0 & A_{22}- XA_{12}  \end{matrix}\right] $$
	if and only if $X$ is a solution of the Riccati equation $A_{22}X - XA_{11}- XA_{12}X+ A_{21} =0.$ \\

	 Hence $\left( \left[ \begin{matrix} I \\  X\end{matrix}\right], \, A_{11}+ A_{12}X\right)$ and  $\left(\left[ \begin{matrix} -X^*  \\ I\end{matrix}\right], \, A_{22}- X A_{12}\right)$ are  right and left invariant pairs of $A$ if and only if  $A_{22}X - XA_{11}- XA_{12}X+ A_{21} =0.$
\end{theorem}

\vone 
\begin{proof}  Define $  \mathbf{R}(X) := A_{22}X - XA_{11}- XA_{12}X+ A_{21}$.  Then we have 
	$$
	\left[ \begin{matrix} I &  0 \\ -X & I  \end{matrix}\right] \left[ \begin{matrix} A_{11} & A_{12} \\ A_{21} & A_{22}   \end{matrix}\right] \left[ \begin{matrix} I &  0 \\ X & I  \end{matrix}\right] = \left[ \begin{matrix} A_{11}+ A_{12}X  &  A_{12} \\ \mathbf{R}(X) & A_{22}- XA_{12}  \end{matrix}\right]. $$ Hence the first assertion holds $\Longleftrightarrow \mathbf{R}(X) = 0.$ 	Again, it follows that $$ A \left[ \begin{matrix} I \\  X\end{matrix}\right] = \left[ \begin{matrix} I \\  X\end{matrix}\right]  (A_{11}+ A_{12}X) \mbox{ and }\left[\begin{matrix} - X & I \end{matrix}\right] A =  (A_{22}- XA_{12})\left[\begin{matrix} - X & I \end{matrix}\right] \Longleftrightarrow \mathbf{R}(X) = 0. $$  Hence the desired results follow.
\end{proof}

\vone  Next, we  investigate the existence of a solution of a Riccati operator equation. Let $ (A, B, C, D)  \in \C^{m\times m} \times \C^{n\times n} \times \C^{n\times m} \times \C^{m\times n}.$  Consider the Riccati operator $\mathbf{R} : \C^{m\times n} \longrightarrow \C^{m\times n}, X \longmapsto AX-XB - X C X +D. $ Then  \, \be\label{reccati}  \mathbf{R}(X) = AX-XB - X C X +D = \left[\begin{matrix} -X & I \end{matrix} \right] \left[\begin{matrix} B & C \\ D & A \end{matrix} \right] \left[\begin{matrix} I \\ X \end{matrix} \right].\ee
Observe that $R(X) = 0$ is a quadratic matrix equation in $X$ and is called algebraic Riccati equation (ARE). 

%
%
%
%
%
%
%

\vone 
\begin{theorem} \label{th:riccati} Suppose that $ s := \sep(A, B) >0$ and that $ 4 \|C\|\, \|D\| < s^2.$ Then there is a unique solution $X$ of the Riccati equation $ \mathbf{R}(X) := AX-XB-XCX + D =0$ such that $$\|X\| \leq \frac{2 \|D\|}{s + \sqrt{s^2- 4 \|C\|\, \|D\|}} \leq \frac{2 \|D\|}{s}.$$
\end{theorem}

\begin{proof} Since $s >0,$ the Sylvester operator  $ \mathbf{T}(X) := AX -XB$ is invertible.  Now define $ \mathbf{S} : \C^{m\times n} \longrightarrow \C^{m\times n}$and $ \phi : [0, \infty) \rar \R$  by $$ \mathbf{S}(X) := \mathbf{T}^{-1}( D + XCX) \mbox{ and }  \phi(x) := \frac{ \|D\| + \|C\| x^2}{s}.$$
Then $\phi$ is strictly increasing and $\|\mathbf{S}(X)\| \leq  \phi(\|X\|)$ for all $ X \in \C^{m\times n}.$  Notice that $X$ is a solution of the Riccati equation $\mathbf{R}(X) = 0 \Longleftrightarrow \mathbf{S}(X) = X $, that is,  if and only if $ X$ is a fixed point of $\mathbf{S}.$ It is easy to see that $\phi$ has two fixed points and the smallest fixed point is given by $$ p :=  \frac{2 \|D\|}{s + \sqrt{s^2- 4 \|C\|\, \|D\|}} \leq \frac{2 \|D\|}{s}.$$ Indeed, we have $\phi(x) = x \Longleftrightarrow q(x) :=\|C\| x^2 - sx + \|D\| =0.$ The quadratic  $q(x)$ has two distinct real roots as the discriminant $ s^2 - 4 \|C\|\, \|D\| >0$ and the smallest root is given by  $p.$

Now consider the closed ball $ B[0, p] := \{ X \in \C^{m\times n} :  \|X\| \leq  p\}.$  Since $ \phi$ is strictly increasing, for $ X \in B[0, p],$ we have $\|X\| \leq p \Longrightarrow  \|\mathbf{S}(X)\| \leq \phi(\|X\|) \leq \phi(p) = p$. This shows that $X \in B[0, p] \Longrightarrow \mathbf{S}(X) \in B[0, p],$ that is, $\mathbf{S} : B[0, p] \longrightarrow B[0, p].$ Now for $X, Y \in B[0, p],$  we have  \beano \| \mathbf{S}(X) - \mathbf{S}(Y)\| &\leq&  \frac{\|XCX-YCY\|}{s}  \leq \frac{ \|C\| (\|X\|+\|Y\|) \|X-Y\|}{s} \\ & \leq & \frac{ 2 p \|C\|}{s} \, \|X-Y\| = \alpha \|X-Y\|,
\eeano where $ \alpha := 2p \|C\|/s \leq 4 \|C\|\, \|D\| /{s^2} < 1.$  This shows that the restriction of $\mathbf{S}$ on the ball $B[0, p]$ is a contraction map. Hence, by Banach fixed point theorem, $\mathbf{S}$ has a unique fixed point  $X \in B[0, p].$
\end{proof}

\vone 

We now utilize Theorem~\ref{th:riccati} to derive a perturbation bound for a right invariant pair of a block upper triangular matrix.   The following bound is  essentially a variant of Stewart's bound~\cite{stewart73,stewbook}.\\

\begin{theorem}\label{th:inv1} Let $ A := \left[ \begin{matrix} A_{11} & A_{12} \\ 0 & A_{22}   \end{matrix}\right]$ and $E :=\left[ \begin{matrix} E_{11} & E_{12} \\ E_{21} & E_{22}   \end{matrix}\right]$ be $n\times n$ matrices partitioned conformably,  where $A_{11}$ is an $m\times m$ matrix.  Then  $  A \left[ \begin{matrix} I_m \\  0  \end{matrix}\right] = \left[ \begin{matrix} I_m \\ 0  \end{matrix}\right]A_{11}.$ 
Suppose that $s_E := \sep(A_{22}+E_{22}, \, A_{11}+E_{11}) >0$ and that $ 4 \|A_{12}+E_{12}\| \, \|E_{21}\| < s_E^2.$ Then there is a unique matrix $X$ such that $(A+E)\left[ \begin{matrix} I \\  X  \end{matrix}\right] = \left[ \begin{matrix} I \\  X \end{matrix}\right] M$ and
$$ \|X\| \leq  \frac{2 \|E_{21}\|}{s_E + \sqrt{s_E^2- 4 \|A_{12}+E_{12}\|\, \|E_{21}\|}} \leq \frac{2 \|E_{21}\|}{s_E},$$ where $M :=(A_{11}+E_{11}) + (A_{12}+E_{12})X.$
\end{theorem}

\begin{proof} Setting $A+E = \left[ \begin{matrix} A_{11}+E_{11} & A_{12}+E_{12} \\ E_{21} & A_{22} +E_{22}  \end{matrix}\right] =: \left[ \begin{matrix} \hat A_{11} & \hat A_{12} \\ E_{21} & \hat A_{22}   \end{matrix}\right],$ it follows that $ (A+E)\left[ \begin{matrix} I \\  X  \end{matrix}\right] = \left[ \begin{matrix} I \\  X \end{matrix}\right] M \Longleftrightarrow \hat A_{11}+ \hat A_{12} X = M$ and $ \hat A_{22}X- X \hat A_{11} - X \hat A_{12} X + E_{21} =0.$ Now by Theorem~\ref{th:riccati} the Riccati equation $\hat A_{22}X- X \hat A_{11} - X \hat A_{12} X + E_{21} =0$ has a unique solution $X$ satisfying the stated bound.
\end{proof}

\vone \begin{remark} \label{rem:angle}  We mention that the bound in Theorem~\ref{th:inv1} holds for any sub-multiplicative matrix norm. Set $\mathcal{X} := \span\left(\bmatrix{ I_m \\ 0} \right)  $ and $ \mathcal{X}_E := \span\left(\bmatrix{ I_m \\ X} \right).$ Then we have  $ A \mathcal{X} \subset \mathcal{X}$ and $ (A+E) \mathcal{X}_E \subset \mathcal{X}_E.$  	 For the special case of $2$-norm and Frobenius norm, by Proposition~\ref{angle}, we have $$ \tan\theta_{\max}(\mathcal{X}, \mathcal{X}_E) \leq \|X\|  \leq  \frac{2 \|E_{21}\|}{s_E + \sqrt{s_E^2- 4 \|A_{12}+E_{12}\|\, \|E_{21}\|}} \leq \frac{2 \|E_{21}\|}{s_E}.$$

Let $ \mathcal{Y} := \span\left( \left[ \begin{matrix} 0 \\  I_{n-m}\end{matrix}\right]\right)$ and $ \mathcal{Y}_E := \span\left( \left[ \begin{matrix} -X^* \\  I_{n-m}\end{matrix}\right]\right)$. Then we have $A^*\mathcal{Y} \subset \mathcal{Y}$ and $(A+E)^*\mathcal{Y}_E \subset \mathcal{Y}_E$, that is, $\mathcal{Y}$ and $ \mathcal{Y}_E$ are left invariant subspaces of $A$ and $ A+E$, respectively. Further, we have  $$\tan \theta_{max}(\mathcal{Y}, \mathcal{Y}_E) \leq \|X\| \leq \frac{2 \|E_{21}\|}{s_E + \sqrt{s_E^2- 4 \|A_{12}+E_{12}\|\, \|E_{21}\|}} \leq \frac{2 \|E_{21}\|}{s_E}. $$ Also, note that $\sin\theta_{\max}(\mathcal{X}, \mathcal{X}_E) \leq  \tan \theta_{max}(\mathcal{X}, \mathcal{X}_E)$.

\end{remark} 

\vone 

Observe that in Theorem~\ref{th:inv1}, the matrix $E$ implicitly appears on both sides of the condition  $ 4 \|A_{12}+E_{12}\| \, \|E_{21}\| < s_E^2$  which may seem to be problematic. However, in view of Proposition~\ref{lipsep}, the implicit dependence can be removed by defining $ \delta_E := \sep(A_{22}, A_{11})- (\|E_{11}\| + \|E_{22}\|)$ and replacing $s_E$ with $\delta_E$  which leads to the following  bound obtained by  Stewart~\cite{stewart73, stewbook}. \\

\begin{corollary}[Stewart] Let $A$ and $ E$ be as in Theorem~\ref{th:inv1}. Define $ \delta_E := \sep(A_{22}, A_{11})- (\|E_{22}\| + \|E_{11}\|).$  If $ \delta_E >0$ and $4 \|A_{12}+E_{12}\| \, \|E_{21}\| < \delta_E^2$ then there is a unique matrix $X$ such that $(A+E)\left[ \begin{matrix} I \\  X  \end{matrix}\right] = \left[ \begin{matrix} I \\  X \end{matrix}\right] M$ and
	\be\label{stewbd} \|X\| \leq  \frac{2 \|E_{21}\|}{\delta_E + \sqrt{\delta_E^2- 4 \|A_{12}+E_{12}\|\, \|E_{21}\|}} \leq \frac{2 \|E_{21}\|}{\delta_E},\ee where $M :=(A_{11}+E_{11}) + (A_{12}+E_{12})X.$
\end{corollary} 

\vone \begin{proof} By Proposition~\ref{lipsep}, we have   $ \delta_E \leq s_E.$ Note that  $$ \|E_{11}\| + \|E_{22}\|  < s \Longrightarrow \delta_E > 0 \Longrightarrow s_E>0.$$ 	
	Since $4 \|A_{12}+E_{12}\| \, \|E_{21}\| < \delta_E^2 \leq s_E^2,$ by Theorem~\ref{th:inv1}, there exists a unique $X$ such that $(A+E)\left[ \begin{matrix} I \\  X  \end{matrix}\right] = \left[ \begin{matrix} I \\  X \end{matrix}\right] M$ and 
$$	\|X\| \leq \frac{2 \|E_{21}\|}{s_E + \sqrt{s_E^2- 4 \|A_{12}+E_{12}\|\, \|E_{21}\|}} \leq  \frac{2 \|E_{21}\|}{\delta_E + \sqrt{\delta_E^2- 4 \|A_{12}+E_{12}\|\, \|E_{21}\|}}.$$ The last inequality holds because $f(x) := 2 \|E_{21}\|/(x + \sqrt{x^2 - 4 \|A_{12}+E_{12}\|\, \|E_{21}\|})$ is monotonically decreasing for $ x^2 > 4 \|A_{12}+E_{12}\| \, \|E_{21}\|$. 	
\end{proof} 

\vone 

Theorem~\ref{th:inv1} holds when $\|E\|$ is sufficiently small. It is desirable to derive a similar bound when $\|E\|$ is as large as possible. A bound similar to that in Theorem~\ref{th:inv1} is derived in~\cite{karow} that involves $s$ instead of $s_E.$ We now derive the bound from a different perspective and provide a simplified proof. We proceed as follows.

\vone 
{\bf Pseudospectrum.} Let $\ep >0$ and $ A \in \C^{n \rtimes n}.$ Then the $\epsilon$-pseudospectrum $\eig_{\epsilon}(A)$ of $A$ is defined by  (see, \cite{alambora, alambora2, lntbook}) $$\eig_{\epsilon}(A) := \bigcup_{\|E\| \leq \ep}\{\eig(A+E) : E
\in \C^{n\times n} \}.$$ Note that $ \eig_0(A) = \eig(A).$ For the $2$-norm and the Frobenius norm, it is easy to see that  $\eig_{\epsilon}(A) = \{ z \in \C : \sigma_{\min}(A- zI) \leq \epsilon\} = \{  z \in \C : \|(A-zI)^{-1} \|_2 \geq 1/\ep\}\cup \eig(A),$ where $\sig_{\min}(A)$ denotes the smallest singular value of $A.$ Further, $\eig_{\ep}(A)$ has the following properties~(see, \cite{alambora, lntbook}). \\

\begin{itemize} \item[(a)]  The map $ \ep \longmapsto \eig_{\epsilon}(A)$ is monotonically increasing, that is,  $ \ep < \delta  \Longrightarrow \eig_{\epsilon}(A) \subset \eig_{\delta}(A).$  Further, $\eig_{\epsilon}(A)$ does not have isolated points for $\ep >0.$
\item[]	
\item[(b)] 	For $\ep >0,$ $\eig_{\epsilon}(A)$
	consists of at most $n$ components (i.e., maximal connected subsets) and each component
	contains at least one eigenvalue of $A$ in its interior. The boundary of $\eig_{\ep}(A)$ is embedded in an algebraic curve.
\end{itemize}

\vone  Let $A \in \C^{n\times n}$ and $B \in \C^{m\times m}.$ We have seen that $\sep(A, B)$ is a measure of separation of $A$ and $B$. We now consider another measure of separation through spectral overlap of $A$ and $B$.  For the rest of this section, we consider only the $2$-norm of matrices. 

\vone \begin{definition}  Let $ A \in \C^{n\times n} $ and $ B \in \C^{m\times m}.$ Define   $$\sep_{\lam}(A, B) := \mathrm{inf}\{ \ep \geq 0: \eig_{\epsilon}(A)\cap\eig_{\epsilon}(B) \neq
	\emptyset\}.$$ 
\end{definition} 

Note that  $ \ep < \sep_{\lam}(A, B) \Longrightarrow \eig_{\epsilon}(A) \cap \eig_{\epsilon}(B) = \emptyset \Longrightarrow \eig(A) \cap \eig(B) = \emptyset.$ The separation $\sep_{\lam}(A, B)$ is a measure of separation of $A$ and $B$  in the sense that $\sep_{\lam}(A, B)$ is the smallest value for which there exist $E \in \C^{n\times n}$ and $ F \in \C^{m\times m}$  such that $\max(\|E\|_2,  \|F\|_2) = \sep_{\lam}(A, B)$ and  $ \eig(A+E) \cap \eig(B+F) \neq \emptyset, $ that is, $\eig(A+E)$ and $\eig(B+F)$ overlap; see~\cite{alambora, alambora2, Dem}. \vone 

\vone \begin{remark}
It is easy to see that $\sep(A, B) \leq 2 \sep_{\lam}(A, B).$  Indeed, assume that $\sep(A, B) > 0.$ Let $E$ and $F$ be such that $\max(\|E\|_2,  \|F\|_2) = \sep_{\lam}(A, B)$ and $ \eig(A+E) \cap \eig(B+F) \neq \emptyset $. Then by Proposition~\ref{lipsep}, we have $$ \sep(A, B) - (\|E\|_2+\|F\|_2) \leq \sep(A+E, B+F) = 0 \Longrightarrow \sep(A, B) \leq 2 \sep_{\lam}(A, B).$$ Consequently, if $ \epsilon < \sep(A, B)/2 $ then $ \eig_{\ep}(A)\cap \eig_{\ep}(B) = \emptyset.$  

\end{remark}

\vone  The  pseudospectrum localization theorem  given below  plays an important role in deriving perturbation bounds for right invariant pairs. \vone 
\begin{theorem}\cite{large} \label{grammont}
Let $ A = Q\bmatrix{ A_1 &
	C
	\\ 0 & A_2}Q^*$, where $ Q$ is unitary.  Define $\phi(\ep) := \ep\sqrt{ 1+ \|C\|_2/\ep}.$ Then   
$\eig_{\ep}(A) \subset \eig_{\phi(\ep)}(A_1)\cup\eig_{\phi(\ep)}(A_2). $
\end{theorem} 

\vone 
The function $\phi$ is strictly increasing and it is easy to see that
\be \label{phi} \phi^{-1}(\ep) = 2 \ep^2/(\|C\|_2+\sqrt{\|C\|_2^2+4 \ep^2}).\ee This
shows that if $\ep < \phi^{-1}(\sep_{\lam}(A_1, A_2))$ then each
component of $\eig_{\ep}(A)$ either intersects $\eig(A_1)$ or
$\eig(A_2)$ but not both.  Since $\sep(A_1, A_2) \leq 2 \sep_{\lam}(A_1,A_2)$, it follows that $  \ep < \phi^{-1}(\sep(A_1, A_2)/2) \Longrightarrow$ each
component of $\eig_{\ep}(A)$ either intersects $\eig(A_1)$ or
$\eig(A_2)$ but not both. \von

A contour $\Gamma$ in $\C$  is called a Cauchy contour if its the oriented boundary of a bounded Cauchy domain~\cite[p.6]{ggk}.

\begin{theorem}\label{proj} Let $ \Delta_1, \ldots, \Delta_m$ be  connected components of $\eig_{\ep}(A)$. Define $\Delta := \cup^m_{j=1}\Delta_j$ and $\Delta_e := \eig_{\ep}(A)\setminus \Delta.$ Then $\Delta$ and $ \Delta_e$ are disjoint compact sets and $\eig_{\ep}(A)	= \Delta \cup \Delta_e.$ 
	Further, there exits a Cauchy contour  $\Gamma$  such that  $\Delta \subset \mathrm{Int}(\Gamma)$ and $ \Delta_e \subset \mathrm{Ext}(\Gamma).$  Let $E \in \C^{n\times n}$ be such that $\|E\|_2 = \ep.$  Then $ \Gamma \subset \rho(A+tE)$ for all $ t \in \mathbb{D} :=\{ t \in \C : |t| \leq 1\}.$ 
Consider the spectral projections $$ P := \frac{-1}{2\pi i}\int_{\Gamma} (A-zI)^{-1} dz \text{ and } P(t) := \frac{-1}{2 \pi i} \int_\Gamma (A+tE -zI)^{-1} dz \text{ for } t \in \mathbb{D}.$$ Then $P(t)$ is analytic in $\mathbb{D}$ and $ \rank(P(t)) = \rank(P)$ for all $ t \in \mathbb{D}.$  Set $$A_Y(t) := (A+tE)_{|R(P(t))} \text{ and } A_Z(t) := (A+tE)_{|N(P(t))}.$$  Then   we have $\eig(A+tE) = \eig(A_Y(t))\cup\eig(A_Z(t)),$
$$\eig(A_Y(t))=\eig(A+tE) \cap \mathrm{Int}(\Gamma) \subset \Delta \text{ and } \eig(A_Z(t))= \eig(A+tE) \cap \mathrm{Ext}(\Gamma) \subset \Delta_e$$ for all $t \in \mathbb{D}$ and $E \in \C^{n\times n}$ with $\|E\|_2 = \ep$.
In particular, $\Delta$ contains exactly the same number of eigenvalues (counting multiplicity) of $A$ and $A+tE$ for all $ t \in \mathbb{D}.$ Further, if $\{\lam(t), \mu(t)\} \subset \eig(A+tE)$ with $\lam(0) \in \Delta$ and $ \mu(0) \in  \Delta_e$ then   $\lam(t) \in \Delta $ and $ \mu(t) \in \Delta_e$ for all $ t \in \mathbb{D}$ and $E \in \C^{n\times n}$ with $\|E\|_2 = \ep.$.
\end{theorem}
\von
\begin{proof} Obviously $\Delta$ and $\Delta_e$ are compact and disjoint and that $\eig_{\ep}(A) = \Delta\cup\Delta_e.$ 	Since $\Delta$ is compact, there exists a Cauchy contour $\Gamma$ such that $\Delta \subset \mathrm{Int}(\Gamma)$ and $ \Delta_e \subset \mathrm{Ext}(\Gamma),$ see~\cite[p.6]{ggk}.

	Evidently, $ \eig(A+tE) \subset \eig_{\ep}(A)$ for $ t \in \mathbb{D}$ and $ E\in \C^{n\times n}$ with $\|E\|_2 = \ep.$   Hence  $ \Gamma \subset \rho(A+tE)$ and that  $\eig(A+tE) \cap \mathrm{Int}(\Gamma) \subset \Delta\; \text{ and }\; \eig(A+tE) \cap \mathrm{Ext}(\Gamma) \subset \Delta_e$ for $t \in \mathbb{D}$ and $E \in \C^{n\times n}$ with $\|E\|_2 = \ep.$ \von

	Set $R(z) := (A-zI)^{-1}$ and $ R(t,z) := (A+tE-zI)^{-1}.$ Note that $ \|R(z)\|_2 < 1/\ep$ for $z \in \Gamma.$ Hence for $ (t, z) \in \mathbb{D}\times \Gamma$,  we have $ \|tER(z)\|_2  \leq |t| \|E\|_2 \|R(z)\|_2< 1$ which shows that
	$$ R(t,z) = R(z) ( I + tE R(z))^{-1} = R(z)\sum^\infty_{k=0} (-ER(z))^k  t^k.$$ Hence $ t \mapsto R(t, z)$ is analytic in $\mathbb{D}$ for each $z \in \Gamma$ and the series expansion converges absolutely and uniformly as $\mathbb{D}\times \Gamma$ is compact. Consequently $$ t \mapsto P(t) =  P + \frac{(-1)^{n+1}}{2\pi i} \sum^\infty_{k=1}\left( \int_{\Gamma} R(z)(ER(z))^k dz\right) t^k $$ is analytic in $\mathbb{D}.$ Observe that $P(0) = P.$ 	Since $P(t)$ is continuous on $\mathbb{D},$ by Proposition~\ref{contp}, we have $ \rank(P(t)) = \rank(P)$ for all $ t \in \mathbb{D}.$

\von	
 Finally, by Theorem~\ref{spd}, we have $ \eig(A_Y(t)) = \eig(A+tE) \cap\mathrm{Int}(\Gamma) \subset \Delta$ and $ \eig(A_Z(t)) = \eig(A+tE) \cap\mathrm{Ext}(\Gamma) \subset \Delta_e$ and that 
	the total algebraic multiplicity of the eigenvalues of $A+tE$  in $\Delta$ is equal to $\rank(P(t))= \rank(P)$ for all $ t \in \mathbb{D}.$ Note that $\lam(t) \in \eig(A_Y(t))$ and $ \mu(t) \in \eig(A_Z(t))$ for all $ t \in \mathbb{D}.$ Hence $\lam(t) \in \Delta$ and $\mu(t) \in \Delta_e$ for all $ t \in \mathbb{D}.$ This completes the proof.	
\end{proof} 

\von \begin{remark} \label{eigpath} Theorem~\ref{proj} shows that if $ \lam(t)$ is an eigenvalue of $A+tE$ then either $ \lam(t) \in \Delta$ for all $ t \in \mathbb{D}$ or  $ \lam(t) \in \Delta_e$ for all $ t \in \mathbb{D}.$ Hence the eigenvalue path $\lam(t)$ cannot crossover from $\Delta$ to $\Delta_e$ and vice-versa. Further, if $\rank(P) = \ell$ then  $\eig(A+tE)\cap \Delta $ contains exactly $\ell$ eigenvalues (counting multiplicity) for all $ t \in \mathbb{D}.$ In other words, the $\ell$ eigenvalues of $A+tE$ in $\Delta$ evolved from the $\ell$ eigenvalues of $A$ in $\Delta$ for all $ t \in \mathbb{D}.$\end{remark}

\von
We now prove an important result which yields perturbation bounds for right and left invariant pairs of a matrix; see~\cite[Theorem~3.1]{karow}. We provide a simplified proof.

\vone \begin{theorem} \label{inclusion}  Let $ A := \left[ \begin{matrix} A_{11} & A_{12} \\ 0 & A_{22}   \end{matrix}\right]$ and $E :=\left[ \begin{matrix} E_{11} & E_{12} \\ E_{21} & E_{22}   \end{matrix}\right]$ be $n\times n$ matrices,  where $A_{11}$ and $ E_{11}$ are $m\times m$ matrices.    Suppose $ s := \sep(A_{22}, \, A_{11})> 0$ and  $ 4 \|E\|_2 (\|E\|_2+\|A_{12}\|_2) < s^2.$ Define   $\phi(\ep) := \sqrt{\ep ( \ep + \|A_{12}\|_2)}$ for $\ep \geq 0.$  Then the following hold.
\begin{itemize} \item[(a)] There exists a unique matrix $X \in \C^{(n-m)\times m}$ such that $$\left[ \begin{matrix} I_m &  0 \\ X & I_{n-m}  \end{matrix}\right]^{-1} \left[ \begin{matrix} A_{11}+ E_{11}&  A_{12}+E_{12} \\ E_{21} &  A_{22} +E_{22}  \end{matrix}\right] \left[ \begin{matrix} I_m &  0 \\ X & I_{n-m}  \end{matrix}\right] =  \left[ \begin{matrix}  M &  A_{12} +E_{12}\\ 0 &  N  \end{matrix}\right],
	$$
	where $M:= A_{11}+E_{11}+  (A_{12}+E_{12}) X$ and $N :=A_{22}+E_{22}- X  (A_{12}+E_{12}).$ Further, 
	$$ \|X\|_2 \leq \frac{2 \|E\|_2}{s + \sqrt{s^2- 4 (\|A_{12}\|_2+ \|E\|_2) \|E\|_2}} \leq \frac{2 \|E\|_2}{s}.$$
	 
	 \item[(b)] $\eig_{\phi(\|E\|_2)}(A_{11})\cap \eig_{\phi(\|E\|_2)}(A_{22}) = \emptyset$. 
	
		\item[(c)]   $ \eig( M) \subset \eig_{\phi(\|E\|_2)}(A_{11}) \mbox{  and }   \eig(N ) \subset \eig_{\phi(\|E\|_2)}(A_{22}).$
		
	\end{itemize} 	
	\end{theorem} 

\begin{proof}  Consider the Riccati operator  $ \mathbf{R}(X) := \left[\begin{matrix} -X & I \end{matrix} \right] (A+E) \left[\begin{matrix} I \\ X \end{matrix} \right].$ Then by Theorem~\ref{lem:ricati}, (a) holds  if and only if the Riccati equation  $ \mathbf{R}(X) = 0$ 	has a unique solution.  Now, we have   \beano \mathbf{R}(X) &=& A_{22}X - X A_{11} + \left[\begin{matrix} -X & I \end{matrix} \right] E \left[\begin{matrix} I \\ X \end{matrix} \right] - X A_{12}X \\ &=& \mathbf{T}(X) + \left[\begin{matrix} -X & I \end{matrix} \right] E \left[\begin{matrix} I \\ X \end{matrix} \right] - X A_{12}X,\eeano where $\mathbf{T}(X) := A_{22}X - X A_{11}$ is a Sylvester operator. This shows that \be \label{fixmap} \mathbf{R}(X) = 0 \iff X = \mathbf{T}^{-1}\left( XA_{12}X - \left[\begin{matrix} -X & I \end{matrix} \right] E \left[\begin{matrix} I \\ X \end{matrix} \right]\right) =: \mathbf{G}(X).\ee In other words, $\mathbf{R}(X) = 0 $ if and only if $ X$ is a fixed point of $ \mathbf{G}(X).$
	
	Next, we show that $\mathbf{G}(X)$ has a unique fixed point. Define $ \psi :[0, \, \infty) \rar \R$ by $ \psi(x) := (x^2\|A_{12}\|_2 + (1+x^2) \|E\|_2)/s.$  Note that $\psi$ is a strictly increasing function and has two fixed points. Indeed,  $\psi(x) = x \iff  q(x) := (\|A\|_2+\|E\|_2) x^2 - s x + \|E\|_2 =0.$ The quadratic $q(x)$ has two real roots because $ s^2 - 4 \|E\|_2 (\|E\|_2+\|A_{12}\|_2) >0.$ The smallest root of $q(x)$ is the smallest fixed point of $\psi(x)$ and is given by $$ p := \frac{2 \|E\|_2}{s + \sqrt{s^2- 4 (\|A_{12}\|_2+ \|E\|_2) \|E\|_2}} \leq \frac{2 \|E\|_2}{s}.$$ Now we have  $\|\mathbf{G}(X)\|_2 \leq \psi(\|X\|_2) \leq \psi(p) = p$ whenever $ \|X\|_2 \leq p.$ This shows that $ X \in B[0, p] \Longrightarrow \mathbf{G}(X) \in B[0, p],$ where $ B[0, p] := \{ X \in \C^{(n-m)\times m} : \|X\|_2 \leq p\}.$
	Hence by Brouwer's fixed theorem, $\mathbf{G}(X)$ has a fixed point in $B[0, p],$ that is, there a matrix $ X \in \C^{(n-m)\times m} $ such that $ \mathbf{G}(X) = X$ and $\|X\|_2 \leq p. $ This proves (a) except for the uniqueness of $X$.
	
The uniqueness of $X$ follows from (b) and (c). Indeed, assume that (b) and (c) hold.  If possible, suppose that 
 $\mathbf{G}$ has two fixed points $X$ and $Y$.    Then writing  $$ XBX- YBY = X B(X-Y) + (X-Y)BY,$$ we have \beano 0 &=& \mathbf{R}(X) - \mathbf{R}(Y) = \left[\begin{matrix} -X & I \end{matrix} \right] (A+E) \left[\begin{matrix} I \\ X \end{matrix} \right] - \left[\begin{matrix} -Y & I \end{matrix} \right] (A+E) \left[\begin{matrix} I \\ Y \end{matrix} \right] \\ &= & \left[\begin{matrix} -X & I \end{matrix} \right] (A+E) \left[\begin{matrix} 0 \\ X -Y \end{matrix} \right] - \left[\begin{matrix} X-Y & 0 \end{matrix} \right] (A+E) \left[\begin{matrix} I \\ Y \end{matrix} \right] \\ &=&  N (X-Y) - (X-Y) M. \eeano
By (c) we have  $\eig(N) \subset \eig_{\phi(\|E\|)}(A_{22})$ and $ \eig(M) \subset \eig_{\phi(\|E\|)}(A_{11})$. By (b)  $N$ and $M$ have disjoint spectra. Hence  $ N(X-Y) - (X-Y) M = 0 \Longrightarrow X - Y =0.$  This proves uniqueness of $X.$

Now we prove (b).  Suppose that $\mathbf{G}(X) = X.$ Then by (a) and Theorem~\ref{grammont}, we have  $\eig(A+E) = \eig( M)\cup  \eig(N )$
and  $$\eig(A+E) \subset \eig_{\|E\|_2}(A) \subset \eig_{\phi(\|E\|_2)}(A_{11})\cup \eig_{\phi(\|E\|_2)}(A_{22}).$$  Since $\phi(\|E\|_2) < s/2 \leq \sep_{\lambda}(A_{11}, A_{22}),$ we have $\eig_{\phi(\|E\|_2)}(A_{11})\cap \eig_{\phi(\|E\|_2)}(A_{22}) = \emptyset.$ This proves (b). 

Finally,   set $\ep := \|E\|_2$ and consider  $ \Delta_1 := \eig_{\ep}(A) \cap \eig_{\phi(\ep)}(A_{11})$ and $ \Delta_2 := \eig_{\ep}(A) \cap \eig_{\phi(\ep)}(A_{22})$. Then $\Delta_1$ and $\Delta_2$ are disjoint compact sets such that $ \eig_{\ep}(A) = \Delta_1 \cup \Delta_2$ and $\Delta_i \subset \eig_{\phi(\ep)}(A_{ii})$ for $ i=1,2.$
 Note that $\eig(A_{11}) = \eig(A)\cap\Delta_1$ and $\eig(A_{22}) = \eig(A) \cap \Delta_2.$ Also note that $ \eig(M)\cup\eig(N) = \eig(A+E) \subset \eig_{\ep}(A) =\Delta_1\cup\Delta_2$. Since  $M \rightarrow A_{11}$ and $N \rightarrow A_{22}$ as $ \|E\|_2 \rightarrow 0$,   by Theorem~\ref{proj} and Remark~\ref{eigpath}, we have $\eig(M) \subset \Delta_1 \subset \eig_{\phi(\ep)}(A_{11})$ and $ \eig(N) \subset \Delta_2 \subset \eig_{\phi(\ep)}(A_{22}).$ This proves (c).\end{proof}

\vone 
The condition  $ 4 \|E\|_2 (\|E\|_2+\|A_{12}\|_2) < s^2$ in Theorem~\ref{inclusion} can be restated as $$ \|E\|_2 < \phi^{-1}(s/2) =  \frac{s^2}{2(\|A_{12}\|_2+\sqrt{\|A_{12}\|_2^2+ s^2})}.$$ This ensures that $\eig_{\phi(\|E\|_2)}(A_{11})$ and  $\eig_{\phi(\|E\|_2)}(A_{22})$ are disjoint which is utilized in the proof of Theorem~\ref{inclusion}.  For simplicity, now denote $ \sep_{\lam}(A_{22}, A_{11}) $ by $\sep_{\lam}$. Then  $ s/2 \leq \sep_{\lam}.$  Also, if $ \|E\|_2 < \phi^{-1}(\sep_{\lam})$ then 
$\eig_{\phi(\|E\|_2)}(A_{11})$ and $\eig_{\phi(\|E\|)}(A_{22})$ are again disjoint, see the discussion after (\ref{phi}). This raises the following question whose solution is not known.

\vone \noindent
{\bf Problem:}  Does Theorem~\ref{inclusion} still hold if the condition $\|E\|_2 < \phi^{-1}(s/2)$ is replaced with $\|E\|_2 < \phi^{-1}(\sep_{\lam})$ and $s$ is replaced with $2 \sep_{\lam}$ in the bound on $\|X\|_2?$

\vone  As a consequence of the proof of Theorem~\ref{inclusion}, we have the following result. \vone 
\begin{corollary}\label{cor:riccati}Let $ A := \left[ \begin{matrix} A_{11} & A_{12} \\ 0 & A_{22}   \end{matrix}\right]$ and $E :=\left[ \begin{matrix} E_{11} & E_{12} \\ E_{21} & E_{22}   \end{matrix}\right]$ be $n\times n$ matrices,  where $A_{11}$ and $ E_{11}$ are $m\times m$ matrices.    Consider the Riccati operator
	$$ \mathbf{R} : \C^{(n-m)\times m} \longrightarrow \C^{(n-m)\times m}, \,\, X \longmapsto  \left[\begin{matrix} -X & I \end{matrix} \right] (A+E) \left[\begin{matrix} I \\ X \end{matrix} \right].$$ Suppose that $s := \sep(A_{22}, \, A_{11}) >0$ and that $ 4 (\|A_{12}\|_2 + \|E\|_2) \, \|E\|_2 < s^2.$ Then the Riccati equation $\mathbf{R}(X) = 0$ has a unique solution  $ X \in  \C^{(n-m)\times m}$ such that 
	$$\|X\|_2 \leq  \frac{2 \|E\|_2}{s + \sqrt{s^2- 4 (\|A_{12}\|_2+ \|E\|_2) \|E\|_2}} \leq \frac{2 \|E\|_2}{s}.$$
\end{corollary}

\vone
Finally, by Theorem~\ref{inclusion}, we have the following perturbation bounds  for right and left invariant pairs of a matrix; see also \cite[Theorem~3.1]{karow}.

\vone \begin{theorem} \label{th:main}  Let $ A := \left[ \begin{matrix} A_{11} & A_{12} \\ 0 & A_{22}   \end{matrix}\right]$ and $E :=\left[ \begin{matrix} E_{11} & E_{12} \\ E_{21} & E_{22}   \end{matrix}\right]$ be $n\times n$ matrices,  where $A_{11}$ and $ E_{11}$ are $m\times m$ matrices.    Suppose $ s := \sep(A_{22}, \, A_{11})> 0$ and  $ 4 \|E\|_2 (\|E\|_2+\|A_{12}\|_2) < s^2.$   Then the following hold.
	\begin{itemize} \item[(a)] There exists a unique matrix $X \in \C^{(n-m)\times m}$ such that $\left( \left[ \begin{matrix} I_m \\  X\end{matrix}\right], \, M\right)$ and  $\left(\left[ \begin{matrix} -X^*  \\ I_{n-m}\end{matrix}\right], \, N \right)$ are  right and left invariant pairs of $A +E, $ respectively, where $M:= A_{11}+E_{11}+  (A_{12}+E_{12}) X$ and $N :=A_{22}+E_{22}- X  (A_{12}+E_{12}).$
		\item[(b)] Let $ \mathcal{X} := \span\left( \left[ \begin{matrix} I_m \\  0\end{matrix}\right]\right)$ and $\mathcal{X}_E := \span\left( \left[ \begin{matrix} I_m \\  X\end{matrix}\right]\right)$. Then we have $ A\mathcal{X} \subset \mathcal{X}$ and $(A+E) \mathcal{X}_E \subset \mathcal{X}_E$. Further,    $$\tan \theta_{max}(\mathcal{X}, \mathcal{X}_E) =\|X\|_2 \leq  \frac{2 \|E\|_2}{s + \sqrt{s^2- 4 (\|A_{12}\|_2+ \|E\|_2) \|E\|_2}} \leq \frac{2 \|E\|_2}{s}.$$ 
	\item[(c)] Let $ \mathcal{Y} := \span\left( \left[ \begin{matrix} 0 \\  I_{n-m}\end{matrix}\right]\right)$ and $\mathcal{Y}_E := \span\left( \left[ \begin{matrix} -X^* \\  I_{n-m}\end{matrix}\right]\right)$. Then  $A^* \mathcal{Y} \subset \mathcal{Y}$ and $(A+E)^* \mathcal{Y}_E \subset \mathcal{Y}_E$. Further, 
	 $$\tan \theta_{max}(\mathcal{Y}, \mathcal{Y}_E)  = \|X\|_2 \leq  \frac{2 \|E\|_2}{s + \sqrt{s^2- 4 (\|A_{12}\|_2+ \|E\|_2) \|E\|_2}} \leq \frac{2 \|E\|_2}{s}.$$
	\end{itemize} 	
\end{theorem}

\begin{proof} By Theorem~\ref{inclusion}, there exists a unique $X \in \C^{(n-m)\times n}$ such that 
$$\left[ \begin{matrix} I_m &  0 \\ X & I_{n-m}  \end{matrix}\right]^{-1} \left[ \begin{matrix} A_{11}+ E_{11}&  A_{12}+E_{12} \\ E_{21} &  A_{22} +E_{22}  \end{matrix}\right] \left[ \begin{matrix} I_m &  0 \\ X & I_{n-m}  \end{matrix}\right] =  \left[ \begin{matrix}  M &  A_{12} +E_{12}\\ 0 &  N  \end{matrix}\right]
$$ and $$ \|X\|_2 \leq  \frac{2 \|E\|_2}{s + \sqrt{s^2- 4 (\|A_{12}\|_2+ \|E\|_2) \|E\|_2}} \leq \frac{2 \|E\|_2}{s},$$ where $M:= A_{11}+E_{11}+  (A_{12}+E_{12}) X$ and $N :=A_{22}+E_{22}- X  (A_{12}+E_{12}).$

By Proposition~\ref{angle}, we have $ \tan \theta_{max}(\mathcal{X}, \mathcal{X}_E) = \tan \theta_{max}(\mathcal{Y}, \mathcal{Y}_E)  = \|X\|_2.$ Hence (b) and (c) follow. 
\end{proof}

\vone 
\begin{remark} For $\mathcal{X}$ and $\mathcal{X}_E$ in Theorem~\ref{th:main}, $\theta_{\max}(\mathcal{X}, \mathcal{X}_E)$ is an acute angle and hence $\sin\theta_{\max}(\mathcal{X}, \mathcal{X}_E)  \leq \tan \theta_{max}(\mathcal{X}, \mathcal{X}_E).$
 Consequently, we have 
$$\sin\theta_{\max}(\mathcal{X}, \mathcal{X}_E) \leq  \tan \theta_{max}(\mathcal{X}, \mathcal{X}_E) \leq  \frac{2 \|E\|_2}{s + \sqrt{s^2- 4 (\|A_{12}\|_2+ \|E\|_2) \|E\|_2}} \leq \frac{2 \|E\|_2}{s}.$$ 
\end{remark}

A different bound for a right invariant pair of $A$ is derived in \cite[Lemma~7.8]{Dem} by reducing $A$ to block diagonal form; see also \cite[Theorem~3.6]{karow} . We now prove a similar but sharper bound  for a right invariant pair of $A$ by reducing $A$ to block diagonal form $\diag(A_{11}, A_{22}).$ Let $A$ and $E$ be as in  Theorem~\ref{inclusion}.  Let $U \in \C^{n\times n}$ be a nonsingular matrix such that $ U^{-1} AU = \diag(A_{11}, \; A_{22}).$ 
Then $ U^{-1}(A+E) U= \diag(A_{11}, \; A_{22}) + U^{-1}EU$ and $ \|U^{-1} EU\|_2 \leq \|U^{-1}\|_2\|U\|_2 \|E\|_2.$  So, it is important to choose $U$ such that $\|U^{-1}\|_2\|U\|_2$ has the smallest value. It is proved in~\cite{Dem, Dem83} that 
$$ \min\{ \|U^{-1}\|_2 \|U\|_2 :  U^{-1} AU = \diag(A_{11},\; A_{22}) \text{ and } \rank(U) = n\} = \|P\|_2 + \sqrt{ \|P\|_2^2-1}, $$  where $ P :=\bmatrix{ I_m & - R\\ 0 & 0 }$ is the spectral projection of $A$ corresponding to $\eig(A_{11})$ and 
$ A_{11} R- RA_{22} = -A_{12}.$  Set $ p := \|P\|_2 = \sqrt{1 + \|R\|_2^2}$ and consider   $ S := \bmatrix{ I_m & R/p\\ 0 & I_{n-m}/p}$. Then 
  we have   $$S^{-1} =  \bmatrix{ I_m & -R\\ 0 & pI_{n-m}}, S^{-1} AS = \diag(A_{11}, \; A_{22})  \text{ and } \|S\|_2 \|S^{-1}\|_2 = p + \sqrt{p^2 -1},$$ see~\cite{Dem, Dem83}. In other words, $S$ is a block diagonalizing similarity transformation of $A$ that has the smallest condition number $\cond(S) := \|S^{-1}\|_2 \|S\|_2.$ 


%
%

\vone We are now ready to prove the following result which improves the bound in  \cite[Lemma~7.8]{Dem} especially in the case when $ p:= \|P\|_2$ is small. Also, compare this result with \cite[Theorem~3.6]{karow}.

\vone \begin{theorem} \label{demmel}  Let $ A := \left[ \begin{matrix} A_{11} & A_{12} \\ 0 & A_{22}   \end{matrix}\right]$ and $E :=\left[ \begin{matrix} E_{11} & E_{12} \\ E_{21} & E_{22}   \end{matrix}\right]$ be $n\times n$ matrices,  where $A_{11}$ and $ E_{11}$ are $m\times m$ matrices.  Let $S$ and $p$ be as above and let $\phi(\ep) := (p+\sqrt{p^2-1})\epsilon.$   Suppose that  $$ s := \sep(A_{22}, \, A_{11})> 0 \text{ and  }  \|E\|_2  < \frac{s}{2(p + \sqrt{p^2 -1})} = \phi^{-1}(s/2).$$   Then  there exists a unique $X \in \C^{(n-m)\times m}$ such that 
	\be \label{blkeq}\left[ \begin{matrix} I_m &  0 \\ X & I_{n-m}   \end{matrix}\right]^{-1}  S^{-1}\left[ \begin{matrix} A_{11}+ E_{11}&  A_{12}+E_{12} \\ E_{21} &  A_{22} +E_{22}  \end{matrix}\right] S \left[ \begin{matrix} I_m &  0 \\ X & I_{n-m}  \end{matrix}\right] =  \left[ \begin{matrix}  M &  \widehat{E}_{12}\\ &  N  \end{matrix}\right].
	\ee  Here $M:= A_{11}+ \widehat{E}_{11}+  \widehat{E}_{12} X$ and $N :=A_{22}+\widehat{E}_{22}- X  \widehat{E}_{12},$ where $$ \bmatrix{ \widehat{E}_{11} \\ \widehat{E}_{21}} = \bmatrix{ E_{11} -R E_{21} \\ pE_{21}} \text{ and } \bmatrix{ \widehat{E}_{12} \\ \widehat{E}_{22}} = \bmatrix{ \displaystyle{\frac{1}{p}} \big[ E_{11}R + E_{12} - R( E_{21} R + E_{22})\big] \\ E_{21} R+ E_{22} }.$$  	Further, $ \displaystyle{\|X\|_2 \leq  \frac{2 (p + \sqrt{p^2 -1}) \|E\|_2}{s}.}$ Furthermore, the following assertions hold.
	
	\begin{itemize}
		\item[(a)]  $\eig_{\phi(\|E\|_2)}(A_{11})\cap \eig_{\phi(\|E\|_2)}(A_{22}) = \emptyset$.
		\item[(b)]	  $ \eig( M) \subset \eig_{\phi(\|E\|_2)}(A_{11}) \mbox{  and }   \eig(N ) \subset \eig_{\phi(\|E\|_2)}(A_{22}).$
	     \item[(c)]  $\left( S\left[ \begin{matrix} I_m \\  X\end{matrix}\right], \, M\right)$ and  $\left((S^{-1})^*\left[  \begin{matrix} -X^*  \\ I_{n-m}\end{matrix}\right], \, N \right)$ are  right and left invariant pairs of $A +E, $ respectively. 
		\item[(d)] Let $ \mathcal{X} :=  \span\left( \left[ \begin{matrix} I_m \\  0\end{matrix}\right]\right)$ and $\mathcal{X}_E := \span\left( S \left[ \begin{matrix} I_m \\  X\end{matrix}\right]\right)$. Then we have $ A\mathcal{X} \subset \mathcal{X}$ and $(A+E) \mathcal{X}_E \subset \mathcal{X}_E$. Further,    $$\tan \theta_{max}(\mathcal{X}, \mathcal{X}_E)   \leq \frac{2 (p+\sqrt{p^2-1})^2 \|E\|_2}{s}.$$ 
		\item[(e)] Let $ \mathcal{Y} :=  \span\left(\left[ \begin{matrix} 0 \\  I_{n-m}\end{matrix}\right]\right)$ and $\mathcal{Y}_E := \span\left( (S^{-1})^*\left[ \begin{matrix} -X^* \\  I_{n-m}\end{matrix}\right]\right)$. Then  we have $A^* \mathcal{Y} \subset \mathcal{Y}$ and $(A+E)^* \mathcal{Y}_E \subset \mathcal{Y}_E$. Further, 
		$$\tan \theta_{max}(\mathcal{Y}, \mathcal{Y}_E)   \leq  \frac{2 (p+\sqrt{p^2-1})^2 \|E\|_2}{s}.$$
	\end{itemize} 	
\end{theorem} 

\vone 
\begin{proof}  
Define $ \widehat{E} := S^{-1} ES.$  Then $\|\widehat{E}\|_2  \leq  (p + \sqrt{p^2 -1})  \|E\|_2 $  and 
\be \label{diagm} S^{-1}(A+E)S = S^{-1} AS + S^{-1}ES=\bmatrix{ A_{11} & 0 \\ 0 & A_{22}} + \bmatrix{ \widehat{E}_{11} & \widehat{E}_{12} \\ \widehat{E}_{21} & \widehat{E}_{22}},\ee  where $ \bmatrix{ \widehat{E}_{11} \\ \widehat{E}_{21}} = \bmatrix{ E_{11} -R E_{21} \\ pE_{21}}$ and $\bmatrix{ \widehat{E}_{12} \\ \widehat{E}_{22}} = \bmatrix{ \frac{1}{p} ( E_{11}R + E_{12} - R( E_{21} R + E_{22})) \\ E_{21} R+ E_{22} }.$

\von	
	
	Note that $ \|E\|_2 < \phi^{-1}(s/2) \Longrightarrow \|\widehat{E} \|_2 < s/2.$ Hence the condition of Theorem~\ref{inclusion} is satisfied thereby proving the existence of $X, M$ and $N$  as well as the bound on $\|X\|_2.$
	
\von 	
	 Since $ \|S^{-1} ES\|_2  \leq  (p + \sqrt{p^2 -1}) \|E\|_2$ and  $A+E = S \diag(A_{11}, A_{22}) S^{-1} + E,$  it follows that 
	 $ \eig_{\ep} (A) \subset \eig_{\phi(\ep)}(A_{11}) \cup \eig_{\phi(\ep)}(A_{22}).$   Now, if $\ep < \phi^{-1}(s/2)$ then  $\phi(\ep) < s/2 $ which implies that $\eig_{\phi(\ep)}(A_{11}) \cap \eig_{\phi(\ep)}(A_{22}) = \emptyset.$ Since $\|E\|_2 < \phi^{-1}(s/2),$ the assertion in (a) follows.
	 
	 As $ \eig(M)\cup \eig(N) = \eig(A+E) \subset \eig_{\|E\|_2}(A) \subset \eig_{\phi(\|E\|_2)}(A_{11}) \cup \eig_{\phi(\|E\|_2)}(A_{22}),$ the assertion in (b) follows from Theorems~\ref{inclusion} (b). 
	 
	\von  
	The assertion in (c) is immediate from (\ref{blkeq}). We now prove (d).  Obviously,  we have $ A\mathcal{X} \subset \mathcal{X}$ and $(A+E) \mathcal{X}_E \subset \mathcal{X}_E$.
	
	Note that  $ S\left[ \begin{matrix} I_m \\  0\end{matrix}\right] = \left[ \begin{matrix} I_m \\  0\end{matrix}\right]$. Hence 
	 $ \mathcal{X} := \span\left( S\left[ \begin{matrix} I_m \\  0\end{matrix}\right]\right) = \span\left( \left[ \begin{matrix} I_m \\  0\end{matrix}\right]\right). $ Next, we have 
	$ U := S \left[ \begin{matrix} I_m \\  X\end{matrix}\right] =  \left[ \begin{matrix} I_m + RX/p \\  X/p\end{matrix}\right].$ Since $ p^2 = 1+ \|R\|_2^2$ we have $$ \|R\|_2\|X\|_2/p \leq  2 \sqrt{p^2-1} (p+ \sqrt{p^2-1}) \|E\|_2/{ps} < 2(p+ \sqrt{p^2-1}) \|E\|_2/{s} < 1 $$ which shows that $(I_m +RX/p)$ is invertible. Set $ Z :=  X  (I_m + RX/p)^{-1} /p.$ Then  $$ U =  \left[ \begin{matrix} I_m + RX/p \\  X/p\end{matrix}\right] =  \left[ \begin{matrix} I_m  \\  X  (I_m + RX/p)^{-1} /p\end{matrix}\right]( I_m + RX/p ) = \left[ \begin{matrix} I_m \\  Z\end{matrix}\right]( I_m + RX/p )$$ shows that $\mathcal{X}_E = \span(U) = \span\left( \left[ \begin{matrix} I_m \\  Z\end{matrix}\right]\right).$  Hence $\tan \theta_{max}(\mathcal{X}, \mathcal{X}_E) =\|Z\|_2.$
	
Since $\|X\|_2 \leq 2 (p+\sqrt{p^2-1}) \|E\|_2/s < 1$ and $ p^2 = 1+\|R\|_2^2,$   we have 
 $$ \|Z\|_2 \leq \frac{\|X\|_2}{p( 1- \|R\|_2 \|X\|_2/p)} \leq \frac{\|X\|_2}{( p- \|R\|_2 )} =  \frac{\|X\|_2}{( p- \sqrt{p^2-1} )}  =\|X\|_2( p+ \sqrt{p^2-1} )$$ which shows that $\|Z\|_2 \leq 2 ( p+ \sqrt{p^2-1} )^2 \|E\|_2/s.$  This proves (d).
 
 \vone 
 Finally, since  $ (I_{n-m} + XR/p)^{-1}X = X(I_m + RX/p)^{-1} = pZ,$ we have 
 \beano  \bmatrix{ -X & I_{n-m}} S^{-1} &=& \bmatrix{ -X & XR + p I_{n-m}} \\ & =& p(I_{n-m} +XR/p) \bmatrix{ -X(I_m + RX/p)^{-1}/p &  I_{n-m}} \\  & =&  p(I_{n-m} +XR/p) \bmatrix{ - Z &  I_{n-m}}. \eeano 
 Hence  	$ V := (S^{-1})^* \left[ \begin{matrix} -X^* \\  I_{n-m}\end{matrix}\right] =  p \left[ \begin{matrix}  -Z^* \\ I_{n-m}  \end{matrix}\right] (I_{n-m} +XR/p)^*$ shows that $$ \mathcal{Y}_E = \span(V) = \span\left( \left[ \begin{matrix}  -Z^* \\ I_{n-m}  \end{matrix}\right]\right)$$ which in turn shows that $ \tan \theta_{max}(\mathcal{Y}, \mathcal{Y}_E) = \|Z\|_2.$ This completes the proof. 
 \end{proof}

\vone 
We end this section with an asymptotic perturbation  bounds for left and right invariant subspaces which can be derived as a consequence of the implicit function theorem.  See~\cite[Lemma~2.3, Corollary~2.4]{karow}.

\vone \begin{theorem}Let $ A := \left[ \begin{matrix} A_{11} & A_{12} \\ 0 & A_{22}   \end{matrix}\right]$ and $E :=\left[ \begin{matrix} E_{11} & E_{12} \\ E_{21} & E_{22}   \end{matrix}\right]$ be $n\times n$ matrices, where $A_{11} $ and $ E_{11}$ are $m\times m$ matrices.   Suppose $ \sep(A_{22}, \, A_{11})> 0$. Then for sufficiently small $\|E\|_2$ there exists $X_E \in \C^{(n-m)\times m}$ such that  $$ \mathcal{X}_E := \span\left( \left[ \begin{matrix} I_m \\  X_E\end{matrix}\right]\right) \text{ and } \mathcal{Y}_E := \span\left( \left[ \begin{matrix} -X_E^* \\  I_{n-m}\end{matrix}\right]\right)$$   are right  and left invariant subspaces of $A+E$, respectively. Further, we have 
\beano \tan \theta_{max}(\mathcal{X}, \mathcal{X}_E) & = & \|X_E\|_2 \leq \frac{ \|E_{21}\|_2}{\sep(A_{22}, A_{11})}  + {\mathcal O}(\|E\|_2^2), \\  \tan \theta_{max}(\mathcal{Y}, \mathcal{Y}_E) &= & \|X_E\|_2 \leq \frac{ \|E_{21}\|_2}{\sep(A_{22}, A_{11})}  + {\mathcal O}(\|E\|_2^2), \eeano where  $ \mathcal{X} := \span\left( \left[ \begin{matrix} I_m \\  0\end{matrix}\right]\right)$ and $ \mathcal{Y} := \span\left( \left[ \begin{matrix} 0 \\  I_{n-m}\end{matrix}\right]\right).$  

\end{theorem} 
\vone 
\begin{proof} Define   $\mathbf{R}(A+E, X) := A_{22}X - X A_{11} + \left[\begin{matrix} -X & I \end{matrix} \right] E \left[\begin{matrix} I \\ X \end{matrix} \right] - X A_{12}X$ for  $ E \in \C^{n\times n}$ and $ X \in \C^{(n-m)\times m}. $ 
	 By Theorem~\ref{lem:ricati},  $\left( \left[ \begin{matrix} I_m \\  X\end{matrix}\right], M\right) $ is a right invariant pair of $A+E$  for some $M \in \C^{m\times m}$ if and only if 
  $\mathbf{R}(A+E, X) = 0$. Note that $ \left( \left[ \begin{matrix} I_m \\  0\end{matrix}\right], A_{11}\right)$ is an invariant pair of $A$ which corresponds to $ E= 0$ and $X =0.$

  Obviously, $\mathbf{R}$ is a holomorphic function of $E$ and $X.$ The derivative of $ \mathbf{R}(A+E, X)$ with respect to $ X$ at $(A, 0)$ is given by  $\partial_X\mathbf{R}(A, 0) = \mathbf{T},$ where $$ \mathbf{T}: X \longmapsto A_{22}X - XA_{11}$$ is the Sylvester operator which, in this case,  is invertible.  Hence by the implicit function theorem, for sufficiently small $\|E\|_2$ there exists $X_E$ such that $ \mathbf{R}(A+E, X_E) = 0$ and $ X_E =  g(A+E)$ for some function $g$ holomorphic in a neighbourhood of $A.$   Now  the first order expansion $ X_E = g(A) + g'(A)E + \mathcal{O}(\|E\|_2^2)$ together  with the fact that $ g(A) = 0 $ and $$ g'(A)E =  - (\partial_X\mathbf{R}(A, 0))^{-1}\partial_E\mathbf{R}(A, 0)E = - \mathbf{T}^{-1}(E_{21})$$ yields 
  $ X_E = - \mathbf{T}^{-1}( E_{21}) + \mathcal{O}( \|E\|_2^2), $ where $g'(A)$ is the derivative of $g$ at $A.$ Since $\sep(A_{22}, A_{11}) = 1/{\|\mathbf{T}^{-1}\|},$ we have $$ \|X_E\|_2 \leq \|E_{21}\|_2/{\sep(A_{22}, A_{11})} + \mathcal{O}(\|E\|_2^2).$$  Hence  the desired results follow. 
\end{proof}

\vone \begin{remark}  A matrix norm $\|\cdot\|$ is called unitarily invariant if $\| UXV\| = \|X\|$ for all matrices $U,V$ and  $X$ of compatible size with $U$ and $V$ being  unitary. We mention that the results presented above also hold for any unitarily invariant  matrix norm $\|\cdot\|$ such that $\|X\|_2 \leq \|X\|$ for all matrices $X$. The condition $\|X\|_2 \leq \|X\|$ ensures that $\|\cdot\|$ is sub-multiplicative, that is, $\|XY\| \leq \|X\| \|Y\|$ for all matrices $X$ and $Y$ of compatible size. In fact, a unitarily invariant matrix norm $\|\cdot\|$ is sub-multiplicative $\Longleftrightarrow \|X\|_2 \leq \|X\|$ for all matrices $X.$   If a matrix norm  $\|\cdot\|$ is unitarily invariant then  $ \|XYZ\| \leq \|X\|_2 \|Y\| \|Z\|_2$ for all matrices $X, Y $ and $Z$ of compatible size. The converse is also true. Indeed, if $ \|XYZ\| \leq \|X\|_2 \|Y\| \|Z\|_2$ for all matrices $X, Y $ and $Z$ of compatible size, then $ \|UXV\| \leq \|U\|_2 \|X\| \|V\|_2  = \|X\|$ for all $X$ and  unitary matrices $U$ and $V$ of compatible size. On the other hand, $\| X \| = \|U^* UXVV^*\| \leq \|U^*\|_2 \|UXV\| \|V^*\|_2 = \|UXV\|$ shows that $ \|UXV\|= \|X\|$ for all unitary matrices $U$ and $V$ of compatible size. Hence $\|\cdot\|$ is  unitarily invariant. 
	The property $ \|XYZ\| \leq \|X\|_2 \|Y\| \|Z\|_2$ is required in the proof of Theorem~\ref{inclusion}(a).
\end{remark}

\vone 
\section{Perturbation of eigenvalues of matrices}
   Let $ A, B \in \C^{n\times n}.$ Then setting $E := B-A$ and $ F := A-B$, we have  $  B =   A +E$ and $   A =  B +F. $ Thus $B$ can be thought of as a perturbation of $A$ and vice-versa. Our aim in this section is to analyze perturbed eigenvalues and derive  bounds. To that end, we briefly discuss the evolution of eigenvalues of  $A(t) := A+tE$ for all $ t \in \C.$  Owing to the
   special nature of the perturbations $A+tE, \, t \in \C,$ it is to
   be expected that the $\ep$-pseudospectrum of $A$ may fail to give
   specific information about the effect of these perturbations on
   the spectral properties of $A.$ \von
   
    It is well known~\cite{baum,kato} that the eigenvalues  of $A+tE$ are branches of one or several analytic functions having
   at most algebraic singularities and are everywhere continuous on
   $\C.$ Further, the number of distinct eigenvalues remains the same
   for all $t \in \C,$ except  for some exceptional points which form
   a closed discrete subset   of $\C.$ Generically, a multiple eigenvalue $\lam$
   of $A$ {\it splits} to form a $(0, \lam)$-group
   eigenvalues, that is, eigenvalues generated from the splitting of
   $\lam$ as the perturbation $A+tE$ is switched on, which, for sufficiently
   small $|t|$, consists of at most $m(\lam)$ eigenvalues of $A+tE,$
   where $m(\lam)$ is the algebraic   multiplicity of $\lam.$  The $(0,
   \lam)$-group is further divided into one or several cycles
   consisting of distinct elements where eigenvalues in each cycle
   can be developed from $\lam$ by Laurent-Puiseux series. 
   
   \von
  
  While analyzing the eigenvalues of $A+tE$ when $ t$ varies $ \C,$  the main 
  focus has been on the eigenvalues of $A$ which
  split as the perturbation $A+tE$ is switched on - which is no
  doubt a dominant case. However, the structure of $E$ may be such
  that some eigenvalues (including the multiplicities) of $A$ may remain
  unaffected for some or all $t \in \C.$  With a view to analyzing these issues, a geometric  framework has been developed in \cite{alambora3, alambora4} which led
  to a decomposition of $\eig(A)$ into three disjoint sets:
  $$ \eig(A) = \sig^{\infty}(A, E) \cup \sig^f(A, E) \cup \sig^u(A,
  E).$$ Here $\sig^{\infty}(A, E), \sig^f(A, E)$ and $\sig^u(A, E),$
  respectively, denote the set of infinitely stable, finitely stable
  and unstable eigenvalues of $A$ with respect to $E.$ The eigenvalues in each set evolve in a distinct manner as
  $t$ varies in $\C.$    Set $ \sig^s(A, E) := \sig^f(A, E) \cup \sig^{\infty}(A,
  E).$
  
  \von 
   An eigenvalue  $\lam \in \eig(A)$ is infinitely stable (in short,
  $\infty$-stable) with respect to $E$ if $ \lam \in \eig(A+tE)$  and 
  the algebraic multiplicity of $\lam$ remains constant for all $ t
  \in \C.$  An eigenvalue $\lam \in \eig(A)$ is called finitely
  stable with respect to $E$ if there is $\ep_{\lam} >0$ such that $
  \lam \in \eig(A+tE)$ and the algebraic multiplicity of $\lam$
  remains constant for all $|t| < \ep_{\lam}$ but the multiplicity
  increases for some $|t| = \ep_{\lam}.$ An eigenvalue $\lam \in
  \eig(A)$ is said to be stable with respect to $E$ if it is either
  finitely stable or $\infty$-stable. An eigenvalue $\lam $ is
  called unstable with respect to $E$ if it is not stable with
  respect to $E$, see~\cite{alambora3, alambora4}.
  
  \vone\begin{remark} \label{sigmau}
  Observe that if $ \lam \in \sig^s(A, E)$ then $ \lam \in \eig(A+tE)$ for all $ t \in \C.$ On the other hand, if  $ \lam \in \sig^u(A,E)$ then  the totality of the eigenvalues of $A+tE$
  generated from the splitting of $\lam$ is known as the  $(0,\lam)$-group~(\cite{baum},\cite{kato}). For $ \lam \in \sig^u(A, E),$ let $\Lam(t)$ denote
  the $(0, \lam)$-group eigenvalues of $A+tE.$ Thus, for perturbation bounds on eigenvalues of $A+tE$, we only need to investigate the eigenvalues of $A$  in $\sig^u(A, E)$ and the corresponding $(0, \lam)$-groups  eigenvalues of $A+tE$.
  \end{remark}

  \vone 
  \begin{exam}\cite{alambora4}\label{ex:st}
  	 Let $A := \diag( A_1, A_2, A_3)$ and $E:= \diag(E_1, E_2,
  		E_3),$ where {\small $$A_1 := \left[ {\ba{cc} 2 & 1\\ 0 & 2\ea}\right],
  		\,\, A_2
  		:= \left[{\ba{cc} -1 & 1 \\ 0 & 4\ea}\right], \,\, A_3 :=
  		\left[{\ba{cc} 3 & 0 \\ 0 & 1\ea}\right], \,E_2
  		:=\left[{\ba{cc} 0 & 1 \\ 0 & 0 \ea}\right], \,\, E_3 :=\left[
  		{\ba{cc} 1 & 1 \\ 1 & 1\ea} \right]$$} and $E_1$ is a 2-by-2 zero
  		matrix. Then the eigenvalues of $A$ are $2,\, -1,\, 4,\, 3, \,1$
  		and the eigenvalues of $A+tE$ are $2,\, -1,\, 4,\, 2+t
  		+\sqrt{1+t^2}, \,\, 2+t - \sqrt{1+t^2}.$

  		This shows that  $\sig^{\infty}(A, E) :=\{ 2\}, \,\, \sig^f(A, E) := \{-1, 4 \}$
  		and $\sig^u(A, E):=\{ 1, 3\}$ and that the effect of the
  		perturbations $A+tE, \, t \in \C,$ on each of these sets is
  		different. Note that the eigenvalue $2$ is 
  		insensitive to the perturbation $A+tE$ in the sense that it
  		remains in the spectrum of $A+tE$ and its algebraic multiplicity
  		remains constant for all $t \in \C.$
  		
  	On the other hand, the eigenvalues in $\sig^f(A, E)$ are also
  	eigenvalues of $A+tE$ for all $t \in \C,$ however, the algebraic multiplicities of these eigenvalues
  	are no longer constants for all $t \in \C.$ For $t= -4/3$ the
  	eigenvalue $1 \in \sig^u(A, E)$ moves  and coalesces with $-1$ at
  	$-1$ thereby increasing the multiplicity of $-1$ from one to two.
  	Similarly, for $t= 3/4$  the eigenvalue $3 \in \sig^u(A, E)$ moves
  	and coalesces with $4$ at $4$ so that the multiplicity of $4$
  	increases from one to two. But the multiplicity of $-1$ (resp.,
  	$4$) remains constant for all $  |t| < 4/3$ (resp., $ |t| < 3/4$
  	). Thus, the eigenvalues in $\sig^f(A, E)$ are {\it finitely
  		stable} with respect to $E$  in the sense that each $\lam \in
  	\sig^f(A, E)$ remains in the spectrum of  $A+tE$ and the algebraic
  	multiplicity remains constant only up to a certain magnitude of
  	$t.$ Once the magnitude of $t$ exceeds this {\it critical } value,
  	an eigenvalue $\mu \in \sig^u(A, E)$ moves and coalesces with
  	$\lam$ at $\lam$ thereby increasing the multiplicity of $\lam.$

  	Finally, as is evident, the eigenvalues in $\sig^u(A, E)$ move with $t$
  	continuously. Thus the eigenvalues in $\sig^u(A, E)$ are {\it
  		unstable} with respect to $E$ in the  sense that each eigenvalue
  	$\lam$ moves with $t$ and/or the algebraic multiplicity of $\lam$
  	changes as the perturbation is switched on.	\,$\blacksquare$\end{exam}
  
  \vone 
  

Let $R(z) := (A-zI)^{-1}$ for $z \in \rho(A).$ Suppose that $z \in \eig(A+tE)$ for some
  $t \in \C.$ Then there exists nonzero $v \in \C^n$ such that
  $(A+tE)v = z v.$ This shows that  $(A-zI)^{-1}Ev
  = -t^{-1}v$ which in turn shows that  $-t^{-1} \in \eig(ER(z))$  when $ z \in \rho(A).$ Hence we have $|t| r_{\sig}(ER(z)) \geq 1.$

  \von
  Note that the set of singularities of  $r_{\sig}(ER(z))$ is a subset of $\eig(A).$ 
  Let $\phi$ denote the unique subharmonic extension~(see,\cite{alambora4}) of $r_{\sig}(ER(z))$ on the domain  $D_{\phi}$ given by 
    $$D_{\phi} := \rho (A) \cup \{ \lam \in \eig (A) : \dm{\limsup_{ z
  		\rightarrow \lam,\, z \in \rho (A)}} \, r_{\sig}(ER(z))
  < \infty\}.$$
 Then $\phi(z)$ is nonconstant on open subsets of
  $D_{\phi}$ unless it is identically equal to zero on $\C,$ see~\cite{alambora4}.  Consider 
  $$ \Lam_{\ep}(A, E) :=  \cup_{|t|
  	\leq \ep} \eig (A+tE) \text{ and } \sig_{\ep} (A,E)
  := \left\{ z \in \C : \phi (z)\geq \ep^{-1}\right\},$$ where it is
  assumed that $\phi (z) = \infty$ for $z \in \C \setminus D_{\phi}.$ Then $ \sig_{\ep}(A, E) \subset \Lam_{\ep}(A, E)$ and the inclusion may be strict. Also, the boundary of $\sig_{\ep}(A, E)$ may contain  eigenvalues of $A$,  see~\cite{alambora4}. The stable and unstable eigenvalues are characterized by $\phi(z)$ and $\sig_{\ep}(A, E).$

  \vone
  \begin{theorem}\cite{alambora4}\label{ra} Let $ \lam \in \sig(A).$ Then we have the
 	following.
 	\begin{itemize} 
 	\item[(a)] $\lam \in \sig^{\infty}(A, E) \Longleftrightarrow \phi(\lam) = 0  \Longleftrightarrow \lam \not\in \sig_{\ep}(A, E)$ for all
 	$\ep > 0.$
 	
 	\item[(b)] $\lam \in \sig^f(A, E)  \Longleftrightarrow  0 < \phi(\lam)< \infty  \Longleftrightarrow \lam \not\in \sig_{\ep}(A, E)$ for some but not all	$\ep>0.$ 
 	
  	\item[(c)] $\lam \in \sig^u(A, E) \Longleftrightarrow \phi(\lam) = \infty  	\Longleftrightarrow \lam \in \sig_{\ep}(A, E)$ for all $\ep  > 0.$
 	\end{itemize} 
 \end{theorem}
 
 \vone
 The decomposition of $\Lam_{\ep}(A, E)$ into disjoint sets given in next result plays a crucial role in the spectral  analysis of $A+tE $ for $ t \in \C.$ It also shows that $\Lam_{\ep}(A,E)$ can have isolated points and such isolated points are  stable eigenvalues of $A.$ \von
  
  \begin{proposition} \cite{alambora4} \label{compo} The set $\sig_{\ep}(A, E)$  has the following properties.
   \begin{itemize}
   	\item[(a)] For $\ep > 0$, we have, $\sig_{\ep} (A,E) =
   	\Lam_{\ep}( A, E ) \setminus \{\lam \in \sig (A) :  \phi(\lam)
   	< \ep^{-1}   \}.$
   	\item[(b)]  If $\sig_{\ep}(A,E) \neq \emptyset$ then
   	$\sig_{\ep}(A, E)$ does not have isolated points and each connected 
   	component contains at least one eigenvalue of $A$ in its interior.
   	\item[(c)] $\sig_{\ep}(A, E)$ is the closure of $\{ z \in \C : \phi(z) > \ep^{-1}\}.$
   	
   \end{itemize}
   \end{proposition}

  \vone
   Given a curve  $\Gamma \subset \rho(A)$, consider the disk  
  $$\partial_{\Gamma} := \{t \in \C: |t| < 1/{(\dm{\max_{z \in \Gamma}}\, r_{\sig}(ER(A,z)))}\}.$$  
  Then it is easy to see that $ \Gamma \subset \rho(A+tE)$ for $ t \in \partial_{\Gamma}.$ Set $ R(t, z) := (A+tE-zI)^{-1}$ for $ z \in \Gamma.$   If $\sig :=\eig(A)\cap\mathrm{Int}(\Gamma)$ is nonempty then the spectral projection $$P(t) := \frac{1}{2\pi i} \int_{\Gamma} R(t, z) dz$$ associated with $A+tE$ and $\sig(t) := 	\eig(A+tE)\cap \mathrm{Int}(\Gamma)$ is analytic in $\partial_{\Gamma}$, see~\cite{chatelin, limbook2}. We prove this result in section~\ref{operator} for bounded linear operators. 
  
  \von
  We need the following result for deriving bounds on the spectral variations of $A.$ Compare this result with \cite[Theorem~3.1]{alambora4}.

  \vone 
  \begin{theorem}\label{th:rafik}
  	 Suppose that $\sig_{\ep}(A, E)$ has $m$ connected components
  	$\Delta_1, \ldots, \Delta_m.$ Let $\sig_j := \eig(A)\cap\Delta_j$
  	and $\sig_0 := \eig(A)\cap\sig_{\ep}(A,E)^c.$  Then $\sig_j(t):=\eig(A+tE)\cap \Delta_j$ is nonempty and the spectral projection $P_j(t)$ associated with $A+tE$ and $\sig_j(t)$ is analytic for all  $|t| \leq \ep$ and $ j=0:m.$  
  	Further, 	$\eig((A+tE)_{|R(P_0(t))}) = \sig_0$ and $\eig((A+tE)_{|R(P_j(t))}) = \sig_j(t) \subset \Delta_j$ for all $|t| \leq \ep$ and $j=1:m.$

  	Furthermore, $P_0(t)+\cdots +P_m(t) = I$ for $|t|\leq \ep$. 	In particular, $\Delta_j$ contains the same number of eigenvalues of $A$ and $A+tE$ for all $|t| \leq \ep$ and $ j=1:m.$
  \end{theorem}
  
  \von
  \begin{proof} If $ \lam \in \sig_0$ then by Proposition~\ref{compo}, $\phi(\lam) < 1/\ep.$ Hence by Theorem~\ref{ra}, we have  $\sig_0 \subset \sig^s(A, E).$  Next, note that  $ \eig(A) = \cup^m_{j=0} \sig_j$ is a disjoint partition of $\eig(A)$. Let $P_j$ be the spectral projection of $A$ corresponding to $\sig_j$ for $j=0:m.$ Then obviously $P_0+\cdots+P_m =I.$

  	By Proposition~\ref{compo}, $\Lam_{\ep}(A, E) = \sig_0 \cup \sig_{\ep}(A, E)$ and $\sig_{\ep}(A,E) = \cup^m_{j=1} \Delta_j$ are disjoint partitions of $\Lam_{\ep}(A, E)$ and  $\sig_{\ep}(A,E)$, respectively, and the components  $\Delta_1, \ldots, \Delta_m$ are compact. Hence there exist Cauchy contour $\Gamma_j \subset \rho(A)$  for $j=0:m$ with the following properties:
  	\begin{itemize}
  		\item $\sig_0 \subset \mathrm{Int}(\Gamma_0)$ and $ \sig_{\ep}(A, E) \subset \mathrm{Ext}(\Gamma_0).$  
  		\item $\Delta_j \subset \mathrm{Int}(\Gamma_j)$ and $ \sig_0 \cup(\cup_{i\neq j} \Delta_i) \subset \mathrm{Ext}(\Gamma_j)$ for $ j=1:m.$ 
  	\end{itemize}
 
 \von 
  Let  $ z \in \Gamma_j.$ Then by Proposition~\ref{compo},  $\phi(z) = r_{\sig}(ER(z)) < 1/\ep.$ This shows that if $|t| \leq \ep$ then  
   $ |t| r_{\sig}(ER(z)) < 1$ which implies that $\{ t \in \C : |t| \leq \ep\} \subset \partial_{\Gamma_j}.$ Hence $ \Gamma_j \subset \rho(A+tE)$ for $ t \in \partial_{\Gamma_j}$ and the spectral projection  $$P_j(t) = \frac{1}{2\pi i}\int_{\Gamma_j} R(t, z) dz$$ is analytic in $\partial_{\Gamma_j}. $  Consequently,  $P_j(t)$ is analytic for $|t| \leq \ep$ and $ j=0:m.$

   \von
    By Theorem~\ref{spd},   $\eig((A+tE)_{|R(P_j(t))}) = \eig(A+tE)\cap  \mathrm{Int}(\Gamma_j)=\eig(A+tE)\cap \Delta_j$ for all $|t| \leq \ep$ and $j=1:m$  and $\eig((A+tE)_{|R(P_0(t))}) = \eig(A+tE)\cap \mathrm{Int}(\Gamma_0).$

   \von
    Next, note that $ \sig^s(A, E) \subset \eig(A+tE) \subset \Lam_{\ep}(A, E)$ for $|t| \leq \ep.$  By Proposition~\ref{compo},   $ \eig(A+tE) \cap \mathrm{Int}(\Gamma_0) =  \Lam_{\ep}(A, E) \cap \mathrm{Int}(\Gamma_0) = \sig^s(A, E) \cap \mathrm{Int}(\Gamma_0) = \sig_0$ for  $|t| \leq \ep.$
   
   \von
   Obviously, $\eig(A+tE) = \sig_0 \cup(\cup_{j=1}^m\sig_j(t))$ is a disjoint partition of $\eig(A+tE)$ for all $|t| \leq \ep.$ Hence $P_0(t)+\cdots+P_m(t) = I$ for all $ |t| \leq \ep.$  Since $P_j(t)$ is analytic, by Proposition~\ref{contp}, $\rank(P_j(t)) = P_j$ for all $ |t| \leq \ep$ and $ j=0:m.$ Hence the number of eigenvalues (counting multiplicity) in $\sig_j(t)$ is $\rank(P_j(t)) = \rank(P_j)$ for all $ |t| \leq \ep$ and $ j=0:m.$ Consequently, each $\Delta_j$ contains the same number of eigenvalues of $A$ and $A+tE$ for all $|t| \leq \ep.$ This completes the proof.   
  \end{proof}

   \vtwo
  \subsection{Bounds on perturbed eigenvalues}   
  
  Let $ A \in \C^{n\times n}.$  Then $ A$ has exactly $n$ eigenvalues (counting multiplicity). Let $ \lam_1(A), \ldots, \lam_n(A)$  denote the $n$ eigenvalues of $A$.  Let $ B \in \C^{n\times n}.$   First, we present two very well known global bounds for the variation of the spectrum $\eig(A)$. To that end, we consider two measures of distance between  the spectra of $A$ and $B$, namely, the Housedroff distance and  the optimal matching distance; see~\cite{stewbook}.\\

\begin{itemize}\item {\bf   Housedroff distance:}  $$ d_H( \eig(A), \eig(B)) := \max\left( \max_{\lam \in \eig(A)}\dist(\lam, \; \eig(B)),  \;\; \max_{\mu \in \eig(B)}\dist(\mu, \; \eig(A))\right).$$

\item {\bf Optimal matching distance:} 
$$ d_m( \eig(A), \eig(B)) := \min_{\eig \in S_n} \left( \max_{1\leq j\leq n} |\lam_j(A) - \lam_{\sig(j)}(B)|   \right),$$ where $S_n$ is the set of all permutations of $\{1, 2, \ldots, n\}.$ 

\end{itemize}


\vone

\begin{lemma}[Hadamard inequality] \label{hada} Let $ A \in \C^{n\times n}.$ Then   $ |\det(A)| \leq \prod^n_{j=1}\|Ae_j\|_2 .$ 
\end{lemma}
 \begin{proof} Let $ A = QR$ be a QR factorization of $A,$ where $Q$ is unitary and $R$ is upper triangular.  Then $ |\det(A)| = |\det(R)| = \prod^n_{j=1}|r_{jj}| \leq \prod^n_{j=1} \|Re_j\|_2  = \prod^n_{j=1}\|Ae_j\|_2.$ \end{proof} 

\vone
\begin{lemma} \label{lem2}  Let	 $ A, B \in \C^{n\times n}.$ Let $\lam \in \eig(A).$ Then 
$$|\det( \lam I - B )| \leq \|A-B\|_2 (\|A\|_2+\|B\|_2)^{n-1}.$$\end{lemma}

\begin{proof}  Let $ U :=\bmatrix{ x_1 & \cdots&  x_n} \in \C^{n\times n}$ be unitary such that $ Ax_1 = \lam x_1.$  
Then   \beano  |\det( B- \lam I)| &\leq& \prod^n_{j=1} \|(B-\lam I)x_j\|_2 \leq \|B-A\|_2 \prod^n_{j=2}\|(B-\lam I)x_j\|_2 \\ &\leq& \|A-B\|_2 (\|A\|_2+\|B\|_2)^{n-1}.   \eeano \end{proof}

\vone 

We  are now ready to prove a global bound due to Elsner~\cite{stewbook} for the variation of the spectrum $\eig(A).$ \vone 

\begin{theorem}[Elsner] \label{elsner} Let $ A, B \in \C^{n\times n}.$  Then   $$ d_H(\eig(A), \eig(B)) \leq (\|A\|_2 + \|B\|_2)^{1-1/n} \| A- B\|_2^{1/n}.$$ \end{theorem}
\begin{proof} 
Let $ \alpha :=\dist(\lam, \eig(B)) = \max_{\mu \in \sig(A)}\dist(\mu, \eig(B)).$  Then $ \alpha \leq |\lam - \mu|$ for $ \mu \in \eig(B).$   Hence by Lemma~\ref{lem2}
$$\alpha^n \leq  \prod_{\mu \in \eig(B)}|\lam-\mu| = |\det( B- \lam I)| \leq \|A-B\|_2 (\|A\|_2+\|B\|_2)^{n-1}$$ 


Now, reversing the role of $A$ and $B$, we obtain the desired bound. \end{proof} 


\vone 
\begin{lemma}\cite{bhatiabd} \label{maxpol} Let $\mathcal{C}$ be a continuous curve in $\C$  with
endpoints $a$ and $b.$ If $p(z)$ is a monic complex polynomial of degree $n,$ then	$\dm{\max_{z \in \mathcal{C}}|p(z)| \geq \frac{|b-a|^n}{2^{2n-1}}.}$ \end{lemma}

\vone

The next result gives the best possible bound for optimal matching distance between $\eig(A)$ and $\eig(B).$   The proof in~\cite{bhatiabd} uses homotopy argument to count the number eigenvalues in a component. We give a slightly different proof using pseudospectrum  thereby avoiding homotopy argument.  Set $ E := B-A$. Consider the pseudospectrum $\sig_{\ep}(A, E)$ and $\Lam_{\ep}(A, E).$ Since $ \sig := \Lam_{\ep}(A, B) \setminus \sig_{\ep}(A, B) \subset \sig^u(A, E)$ and $ \sig \subset \eig(A+tE)$ for all $ t \in \C,$  for the variation  of $\eig(A)$, we only need to consider the components of $\sig_{\ep}(A, E).$ Recall that $ \Lam_{\ep}(A, E) = \cup_{|t|\leq \ep}\eig(A+tE)$ and $\sig_{\ep}(A, E) \subset \Lam_{\ep}(A, E).$
\vone  

\begin{theorem}[Bhatia-Elsner-Krause, \cite{bhatiabd}]\label{bhatia} Let $ A, B \in \C^{n\times n}.$  Then  $$ d_m(\eig(A), \eig(B)) \leq 4(\|A\|_2 + \|B\|_2)^{1-1/n} \| A- B\|_2^{1/n}.$$ \end{theorem}
	\begin{proof} Set $ E := B-A$ and consider $A(t) := A+tE$ for $|t| \leq 1.$  Let $\Delta $ be a connected component of the $1$-pseudospectrum $\sig_{1}(A, E)$ (i.e., $\sig_{\ep}(A, E)$ with $\ep =1$). Then by Theorem~\ref{th:rafik},
$\Delta$ contains the same number of eigenvalues of $A$ and $A+E =B.$	
	

\vone
Let $ a, b \in \Delta.$ Then by Lemma~\ref{maxpol} there exits $ \lam \in \Delta $ such that  $ |\det(\lam I - A)| \geq 2^{1-2n}|a-b|^n.$  Hence $ \lam \in \eig(A(t_0)$ for some $ |t_0| \leq 1.$ 
W.L.O.G. assume that $ \|A\|_2 \leq \|B\|_2.$   Then by Lemma~\ref{lem2}, we have    $$ |\det(\lam I - A)| \leq (\|A\|_2 + \|A(t_0)\|_2)^{n-1} \| A- A(t_0)\|_2 \leq (\|A\|_2+\|B\|_2)^{n-1} \|A-B\|_2.$$ This shows that
\beano|a- b| &\leq&  2^{(2n-1)/n} \|A-B\|_2^{1/n} ( \|A\|_2+\|B\|_2)^{1- 1/n}\\  &\leq&  4 \|A-B\|_2^{1/n} ( \|A\|_2+\|B\|_2)^{1- 1/n}.  \eeano Note that the components of $\sig_1(A, E)$ provide optimal matching of eigenvalues of $A$ and $B$. Hence the desired result follows. 
\end{proof}

\vone 
The upper bounds in Theorem~\ref{elsner} and Theorem~\ref{bhatia} are sharp in the sense that $\mathcal{O}\left( \|A-B\|^{1/n}\right)$ is the best possible.  We illustrate this fact by an example. 

\vone

\begin{exam} Let $ \ep \geq 0.$   Consider an $n\times n$ Jordan block $A$ and its perturbation $A(\ep)$ given by 
$$   A := \bmatrix{2 & 1 &  & \\ & \ddots & \ddots & \\ & & \ddots & 1 \\  & & & 2}  \in \C^{n\times n} \; \text{ and } \; A(\ep) := \bmatrix{2 & 1 &  & \\ & \ddots & \ddots & \\ & & \ddots & 1 \\ \ep & & & 2}  \in \C^{n\times n}.$$ Then   $p_{\ep}(x) := \det(x I- A(\ep) ) = (x-2)^n  + \ep$ shows that $\lam :=2$ is the eigenvalue of $A(0)$ of multiplicity $n$ and $A(\ep)$ has $n$ distinct eigenvalues 
$$ \lam_j(\ep) := 2 +  \ep^{1/n}\, e^{ {(2j-1) \pi i}/{n}}, \;  \,  j=1:n,  \text{ when  } \ep >0.$$ Consequently, we have  $ |\lam_j(\ep) -2| = \ep^{1/n}$ for $ j=1:n.$  Hence $$\d_{H}(\eig(A(0), \eig(A(\ep)) = d_m(\eig(A(0), \eig(A(\ep)) = \ep^{1/n}.$$ 

\begin{figure} [h!]
	\begin{center}	
		\includegraphics[scale= .5]{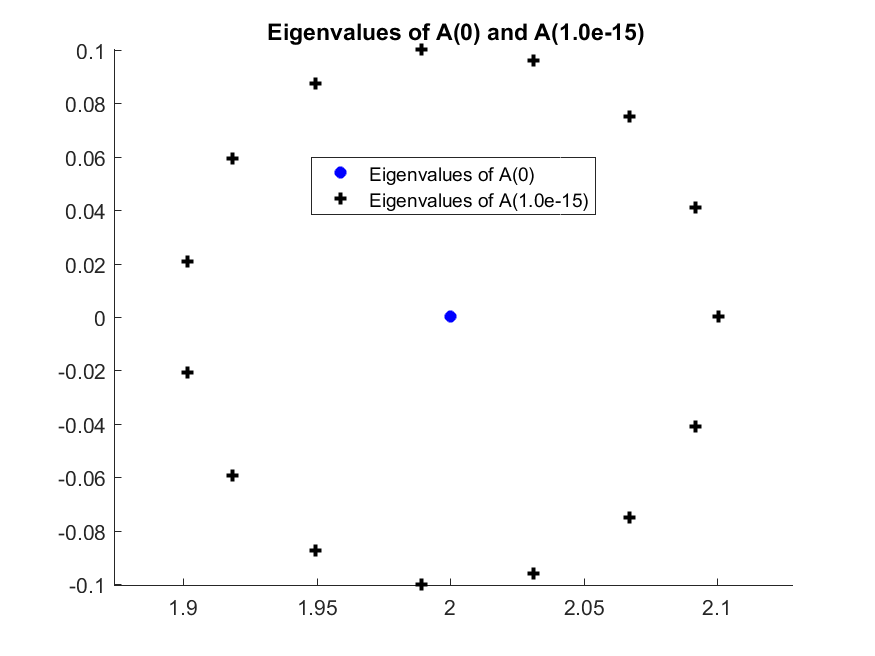}
		\caption{Eigenvalues of Jordan block $A(0)$ and perturbed matrix $ A(10^{-15}).$ }
	\end{center}
\end{figure}

Now, for  $ n=15$ and $ \ep = 10^{-15},$  we have $ |\lam_j(10^{-15}) - 2| = 10^{-15/{15}} = 10^{-1}$ for $j=1:n.$  This shows that the error $10^{-15}$ in the $(15, 1)$ entry of  $A(0) =A$ is magnified  $10^{14}$ times in the eigenvalues of $A(10^{-15}).$

\end{exam} 

\vone

 As the example above shows, the bounds in Theorems~\ref{elsner} and~\ref{bhatia} are sharp when $A$ has a Jordan block of size $n$. Otherwise, these bounds may be gross overestimates of the variation of $\eig(A).$  A better strategy is to derive bounds for variation of individual eigenvalues of $A.$ We now derive global bounds for variation of an eigenvalue of $A.$

\vone 
Let $ A = Q TQ^*$ be a Schur decomposition of $A,$ where $T$ is upper triangular. Then $ T = D + N$ where $ D$ is diagonal and $N$ is nilpotent.   
Thus $$ A = Q( D+N)Q^*.$$ Define {\em  departure from normality } $\delta(A)$ by $ \delta(A) := \|N\|_2.$ 

\vone \begin{theorem}[Henrici, \cite{stewbook}] \label{hen} Let $ A, E \in \C^{n\times n}$ and $ \mu \in \eig(A+E).$   Then there is a $ \lam \in \eig(A)$ such that 
$$ \frac{ \left( \frac{ |\lam - \mu|}{\delta(A)} \right)^n} { 1+ \frac{ |\lam - \mu|}{\delta(A)} + \cdots + \left(\frac{ |\lam - \mu|}{\delta(A)}\right)^{n-1} }  \leq \frac{ \|E\|_2}{\delta(A)}.
$$	
\end{theorem} 

\begin{proof} WLOG assume that $ \mu \notin \eig(A).$ Since $$ A-\mu I +E = (A-\mu I) (I + (A-\mu I)^{-1} E)$$ is  singular and  $ Q^* AQ = D+N$, it follows that 
$ 1/\|(D+N-\mu I)^{-1}\|_2 \leq \|E\|_2.$ 

\von  Set $ \alpha := |\lam - \mu| = \dist(\mu, \eig(A)).$ Now
$$  (D- \mu I + N) ^{-1} = \sum^{n-1}_{k=0} [(\mu I-D)^{-1} N]^k (D- \mu I)^{-1} $$ yields  
$\| (D + N - \mu I)^{-1} \|_2 \leq \frac{1}{\alpha} \sum^{n-1}_{k=0} \left( \frac{ \delta(A)}{\alpha} \right)^k.$ 

Next, $\frac{1}{\alpha} \sum^{n-1}_{k=0} \left( \frac{ \delta(A)}{\alpha} \right)^k =  \frac{1}{\alpha}  \left(\frac{ \delta(A)}{\alpha} \right)^{n-1} \sum^{n-1}_{k=0} \left(  \frac{\alpha}{\delta(A)}\right)^k $ yields 
$$ \|E\|_2 \geq \|(D +N- \mu I)^{-1}\|_2^{-1} \geq  \frac{\frac{ \alpha ^n}{\delta(A)^{n-1}}}{ \sum^{n-1}_{k=0} \left(  \frac{\alpha}{\delta(A)}\right)^k}$$which gives the desired bound.
\end{proof}

\vone 

Define $\phi_m(x) :=1 + x+ \ldots + x^{m-1}$  for  $ x >0$. Then $ \phi_m(x)  \leq (1 +x)^{m-1} $ for all $ x > 0.$  Set $ \kappa(\lam, \mu, A) := \delta(A)^{m-1}\phi_m\left( |\lam- \mu|/{\delta(A)}\right).$ Then by Theorem~\ref{hen}, we have the following result. 
 \vone 
 
 \begin{corollary} 
 	 Let $ A, E \in \C^{n\times n}$ and $ \mu \in \eig(A+E).$  Then there is a $ \lam \in \eig(A)$ such that 
 	$$ \frac{ \left( \frac{ |\lam - \mu|}{\delta(A)} \right)^n} { \left(1+ \frac{ |\lam - \mu|}{\delta(A)}\right)^{n-1}}  \leq \frac{ \|E\|_2}{\delta(A)} \; \text{ and } \; |\lam-\mu|^n \leq \kappa(\lam, \mu, A) \|E\|_2.
 	$$	
 \end{corollary}

 Let $ X\in \C^{n\times n}$ be such that $ A = X J X^{-1}$, where $ J$ is the Jordan canonical form (JCF) of $A.$ Then $\delta(J) = 1$. Hence Theorem~\ref{hen} applied to $J$ yields the following result. \vone 
 
 \begin{corollary}  Let $ X^{-1} A X =  J $ be the JCF of $A$ and  $m$ be the size of the largest Jordan block in $J.$ Let $ \mu \in \eig(A+E).$ Then there is a $ \lam \in \eig(A)$ such that 
 	$$ \frac{   |\lam - \mu|^m}{ \left(1+  |\lam - \mu| \right)^{m-1}} \leq  \frac{  |\lam - \mu|^m}{ \left(1+  |\lam - \mu| + \cdots + |\lam- \mu|^{m-1}\right)}  \leq \|X^{-1}E X\|_2.
 	$$	Further, setting $\kappa(\lam, \mu, A) := \phi_m(|\lam-\mu|),$ we have  $|\lam-\mu|^m \leq \kappa(\lam, \mu, A) \|X^{-1}EX\|_2.$
 \end{corollary}

\vone 

The bound $  \displaystyle {\frac{   |\lam - \mu|^m}{ \left(1+  |\lam - \mu| \right)^{m-1}} \leq \| X^{-1} EX\|_2}$ is proved in~\cite{kahan} by following an entirely  different method and the bound $ \displaystyle{\frac{  |\lam - \mu|^m}{ \left(1+  |\lam - \mu| + \cdots + |\lam- \mu|^{m-1}\right)}  \leq \|X^{-1}E X\|_2}$ is deduced from Henrici theorem in \cite[Theorem~1.12, p.174]{stewbook}. For $m > 1,$ all these bounds suffer from the same defects, namely, that the bounds hold only when $|\lam-\mu|^m$ is divided by  $\phi_m(|\lam-\mu|) $ and that $\|X^{-1}EX\|_2$ is eigenvalue agnostic, that is, $\|X^{-1} EX\|_2$ can be very large irrespective of the ascent $m$ of the eigenvalue $\lam$.

\vone 
Of course, better bounds can be obtained under additional assumptions on $A.$  For instance, if $A$ is normal, that is,  if $AA^* = A^* A$ then we have the following result for the Frobenius norm  $\|A\|_F := \sqrt{ \tr(A^*A) }.$

\vone \begin{theorem}[Hoffman-Wielandt, \cite{bhatiabook, stewbook}] Let $A, B \in \C^{n\times n} $  be normal with eigenvalues $\lam_1, \ldots, \lam_n$  and $\mu_1, \ldots, \mu_n$, respectively. Then there
exists a permutation $ \tau \in S_n$  such that $$ \sqrt{\sum^n_{j=1} |\lam_j - \mu_{\tau(j)}|^2} \leq  \|A-B\|_F.$$ \end{theorem}

If we relax the normality assumption on $B$ then the following result holds. \vone 

\begin{theorem}[Sun, \cite{sun4}]  Let $A, B \in \C^{n\times n} $ with  eigenvalues $\lam_1, \ldots, \lam_n$  and $\mu_1, \ldots, \mu_n$, respectively. If $A$ is normal then 
there
exists $ \tau \in S_n$  such that $$ \sqrt{\sum^n_{j=1} |\lam_j - \mu_{\tau(j)}|^2} \leq  \sqrt{n} \; \|A-B\|_F.$$
Also, $d_m(\eig(A), \eig(B)) \leq n \|A-B\|_2.$
\end{theorem}

\vone 
For simple eigenvalues, asymptotic and non-asymptotic local bounds provide sharp estimates of the errors in the perturbed eigenvalues. Let $(\lam, u, v)$ be a simple eigentriple of $A.$ Then $\lam$ is a simple zero of $p(z) := \det(zI-A),$ that is, $p'(\lam) \neq 0,$ where $p'(z)$ is the derivative of $p(z).$


 Define $ F : \C^{n\times n} \times \C \longrightarrow \C$ by $ F(X, z) := \det(zI - X).$  Since $\lam$ is a simple eigenvalue of $A$, we have $ \partial_zF(A, \lam) = p'(\lam) \neq 0$, where   $ \partial_zF(A, \lam)$ is the partial derivative of $F$ with respect to $z$ evaluated at $(A, \lam).$ Therefore, the implicit function theorem together with Jocobi formula for the derivative of the determinant yields the following result, see~\cite{alamela} for details. 

\von
\begin{theorem}\cite{alamela} Let $(\lam_A, u, v)$ be a simple eigentriple of $A.$ Then  there is an open set  $\Omega \subset \cnn$ containing $A$  and a
	smooth function $\lam: \Omega \rar \C$ such that $\lam(A) = \lam_A$ and $\lam(X)$ is a simple eigenvalue of $X$ for all
	$X \in \Omega.$ Further,  we have  \beano \lam(A+\Delta A) &=& \lam(A) + \inp{\Delta A}{\nabla\lam(A)} + {\cal O}(\|\Delta A\|_2^2) \\
	&=& \lam(A) + \inp{\Delta A}{uv^*/v^*u} + {\cal O}(\|\Delta A\|_2^2)\eeano
	and the first order bound  $|\lam(A+\Delta A) - \lam(A)| \lesssim \cond(\lam, A) \|\Delta A\|_2$ for small $\|\Delta A\|_2$, where $ \nabla \lam(A)$ is gradient and $\cond(\lam, A) := \|u\|_2\|v\|_2/{|u^*v|}$ is the condition number of $\lam.$ 
	
\end{theorem}

\vone 

A non-asymptotic local bound of a simple  eigenvalue of $A$ can be obtained from and Theorem~\ref{demmel} as follows. Let $ (\lam, u, v)$ be a simple eigentriple of $A.$ Then there is a unitary matrix $U$ such that $ A= U \bmatrix{ \lam & c\\ 0 & A_2} U^*$ and $\sep(\lam, A_2) = \sig_{\min}(A_2-\lam I) >0,$ where $\sig_{\min}(X)$ is the smallest singular value of $X.$ Further,  $ P := vu^*/{u^*v}$ is the spectral projection of $A$ corresponding to $\lam$. Set $ p:= \|P\|_2.$ We have seen in  the discussion leading to Theorem~\ref{demmel} that  
$$ \kappa := \min\{ \|S\|_2\|S^{-1}\|_2 :  S^{-1}AS = \diag(\lam, A_2)\}  = p+\sqrt{p^2-1}$$  and $\eig_{\ep}(A) \subset \eig_{\phi(\ep)} (\diag(\lam, A_2)) = B[\lam, \phi(\ep)] \cup \eig_{\phi(\ep)}(A_2),$ where $ \phi(\ep) := \kappa \ep$ and $B[\lam, r] :=\{ z \in \C : |\lam-z| \leq r\}.$ Consequently,  if $ \ep < \phi^{-1}(\sep(\lam, A_2)/2)$ then the disk $B[\lam, \phi(\ep)]$ is disjoint from $\eig_{\phi(\ep)}(A_2).$ Hence  $ A+E$ has a simple eigenvalue $\lam_E$ in $B[\lam, \phi(\|E\|_2)] $ when $ \|E\| < \phi^{-1}(\sep(\lam, A_2)/2).$ Therefore, by Theorem~\ref{demmel} we have the following result. 

\vone 
\begin{theorem} If $ \|E\|_2 < \frac{ \sep(\lam, A_2)}{2 (p+\sqrt{p^2-1})}$ then $ A+E$ has a simple eigenvalue $\lam_E$ such that $$|\lam - \lam_E| \leq  (p+\sqrt{p^2-1}) \|E\|_2.$$ Further, if $v_E$ is an eigenvector of $A+E$ corresponding to $\lam_E$ then 
	$$ \tan\theta( v, v_E) \leq \frac{ 2 (p+\sqrt{p^2-1}) \|E\|_2}{\sep(\lam, A_2)},$$ where $\theta(v, v_E)$ is the acute angle between $v$ and $v_E.$
\end{theorem}

\section{Perturbation of discrete eigenvalues}\label{operator}
Let $ A \in BL(X)$ and assume that the discrete spectrum $ \sig_d(A) $ is nonempty. For $ V \in BL(X)$,  consider the one parameter family of operators $ A(t) := A + t V$ for $ t \in \C.$ We analyze the effect of the perturbation $A(t)$ on the discrete eigenvalues of $A$ when $t$ varies in $\C.$ 

%
%
%
%
%
%
%

\vone 
Consider the resolvent operator   $ R(z) := (A- zI)^{-1} $ for $ z \in \rho(A)$.  Then we have the resolvent identity   $ R(z) - R(w) = (z-w) R(z) R(w)$  for  $ z, w \in \rho(A).$ Further, we have the following result. 

\vone
\begin{proposition}\cite{limbook2, chatelin} \label{propop} Let $ A \in BL(X).$ Then the following results hold.
\begin{itemize}  
	

	\item[(a)] Let $ B \in BL(X).$ If $A$ is invertible and $ r_{\sig}((A-B)A^{-1}) < 1$ then $B$ is invertible. Further, we have   
$$ B^{-1} = A^{-1} \sum^\infty_{j=0}  ((A-B)A^{-1} )^j = \sum^\infty_{j=0}  (A^{-1}(A-B) )^j A^{-1}.$$ 

	\item[(b)] Let $ z \in \rho(A)$ and $ B \in BL(X).$ If  $ r_{\sig}((A-B) R(z)) < 1$ then $ z \in \rho(B)$ and $$ (B-z I)^{-1} =  R(z) \sum^\infty_{j=0} [(A-B) R(z)]^j.$$
	
	
	\item[(c)] Let $ E \subset \rho(A)$ be closed and $ B \in BL(X).$ If  $  \max_{z\in E} \|(A-B)R(z)\| < 1$ then $ E \subset \rho(B).$

\end{itemize} 
\end{proposition} 
\vone

\begin{remark} As a consequence of Proposition~\ref{propop}(c), it follows that the map $ A \longmapsto \sig(A)$ is upper semicontinuous. Indeed,  let $U \subset \C$ be an open set such that 
 $ \sig(A)\subset U$.  Then $ E := U^c$  is closed and by Proposition~\ref{propop} there is a $\delta >0$  such that $ \|A-B\| < \delta \Longrightarrow U^c \subset \rho(B)$ which in turn implies that $ \sig(B) \subset U.$ This shows that the set-valued map  $ A \longmapsto \sig(A)$ is upper semicontinuous. 

\end{remark} 

\vone

%
%
%

Now consider the one parameter family   $ A(t) := A + t V$ for $ t \in \C.$ Set  $ R(t, z) := ( A(t)-zI)^{-1}$ for $ z \in \rho(A(t)).$  
The function $ z \longmapsto r_{\sig}(VR(z))$ is upper semicontinuous (in fact, subharmonic)  and hence attends maximum on a compact subset of $\rho(A)$, see~\cite{chatelin, limbook2}.

\vone \begin{theorem} \cite{limbook2}\label{hol} Let $ t_0 \in \C$ and fix $ z \in \rho(A(t_0)).$ If $ |t-t_0| r_{\sig}(VR(t_0, z)) < 1 $ then $  z\in \rho(A(t))$ and  $$ R(t, z) = R(t_0, z) \sum^{\infty}_{n=0} [-VR(t_0, z)]^n (t-t_0)^n.$$ 
Thus, the function $ t \longmapsto R(t, z)$ is analytic in a neighbourhood $t_0$ for every fixed $ z \in \rho(A(t_0)).$ 

\vone
Let $E \subset \rho(A(t_0))$ be compact. Then the series for $ R(t, z)$ converges uniformly for $ z \in E$ and $t$ in any closed subset of the disk   $$ \mathbb{D}_E :=\{ t \in \C :  |t- t_0| < 1/{ \max_{z \in E} r_{\sig}(VR(t_0, z))}\}.$$
\end{theorem} 

%
%

For $ \Gamma \subset \rho(A),$ define the disk  $$\partial_{\Gamma} := \{ t \in \C : |t|  \max_{z\in \Gamma} r_{\sig}(VR(z)) < 1\}.$$
Then by Theorem~\ref{hol}, we have $ \Gamma \subset \rho(A(t))$ for all $ t \in \partial_{\Gamma}.$ Consider the spectral projections $$ P := \frac{-1}{2\pi i} \int_{\Gamma} R(z) dz\; \text{ and } \; P(t) := \frac{-1}{2\pi i} \int_{\Gamma} R(t, z) dz.$$

\vone

\begin{theorem}[Kato-Rellich, \cite{limbook2,chatelin}] \label{katorel}  Let $ \Gamma \subset \rho(A).$ Then $ \Gamma \subset \rho(A(t))$ for all $ t \in \partial_{\Gamma}.$ The spectral projection 
 $$ P(t) := \frac{-1}{2\pi i} \int_{\Gamma} R(t, z) dz,$$ is analytic in $\partial_{\Gamma}$. For $ t \in \partial_{\Gamma},$  we have the Kato-Rellich perturbation series
$$P(t) = P + \sum^\infty_{n=1} \widehat{P}_n t^n,\;\;  \mbox{ where } \widehat{P}_n = \frac{(-1)^{n+1}}{2\pi i} \int_{\Gamma} R(z) [VR(z)]^n dz.$$

\end{theorem} 

\begin{proof} 
Let $ t_0 \in \partial_{\Gamma}.$ Then letting $E = \Gamma$ and for $t$ in a neighbourhood of $t_0,$ by Theorem~\ref{hol}, the series 
$ R(t, z) = R(t_0, z) \sum^{\infty}_{n=0} [-VR(t_0, z)]^n (t-t_0)^n$ converges uniformly for $ z \in \Gamma.$ Hence term by term integration yields
$$ P(t) = \frac{-1}{2\pi i} \sum^\infty_{n=0} \left( \int_{\Gamma} R(t_0, z) [-V R(t_0, z)]^n dz\right) (t-t_0)^n$$for $t$ near $t_0.$ This shows that $ t \longmapsto P(t)$ is analytic at $ t_0.$ Since $t_0$ is arbitrary, $P(t)$ is analytic in $\partial_{\Gamma}.$  The power series of $P(t)$  follows immediately. \end{proof}

%
%
%
%

\vone


\begin{theorem} \cite{limbook2, chatelin, kato}\label{main} Suppose that $ \sig_0 := \sig(A)\cap \mathrm{Int}(\Gamma) \subset \sig_d(A)$ and that  $\rank(P) = m,$ where $P$ is the spectral projection of $A$ corresponding to $\sig_0.$ Let $ \lam_1, \ldots, \lam_m$ be the $m$ eigenvalues (counting multiplicity) in $\sig_0$.	
	Then  $A(t)$ has exactly $m$ discrete eigenvalues (counting multiplicity)  $\lam_1(t), \ldots, \lam_m(t)$  inside $\Gamma$ for all $t \in \partial_{\Gamma}.$  Define  $ \dm{\lam_{\mathrm{av}} := \frac{\lam_1+\cdots +\lam_m}{m}}$ and  $ \dm{\lam_{\mathrm{av}}(t) = \frac{\lam_1(t)+\cdots +\lam_m(t)}{m}}$ for all $ t \in \partial_\Gamma.$

	\von Then $\dm{ \lam_{\mathrm{av}}(t) =  \frac{\tr(A(t)P(t))}{m}}$ is analytic in $\partial_{\Gamma}$  and
\beano \lam_{\mathrm{av}}(t) &=&  \frac{\tr(A(t)P(t))}{m} = \frac{1}{m} \tr\left( (A+ t V)(P + \sum^\infty_{n=1} \widehat{P}_n t^n)\right)  \\
&=& \lam_{\mathrm{av}} + \sum^\infty_{n=1} \alpha_n t^n \;\; \text{ for } t \in \partial_{\Gamma}.\eeano In particular, if $m=1,$ then $\lam_{\mathrm{av}}(t)$ is a simple eigenvalue of $A(t)$ with $\lam_{\mathrm{av}}(0) = \lam_{\mathrm{av}}$ for $ t \in \partial_\Gamma.$ Further, if $ Av = \lam_{\mathrm{av}} v$ then $ A(t)P(t)v = \lam_{\mathrm{av}}(t) P(t)v$ for $ t \in \partial_{\Gamma}.$ Hence $$  v(t) := P(t) v = v+ \sum^\infty_{n=1} \widehat{v}_n t^n \; \mbox{ for } \; t \in \partial_{\Gamma}.$$

\end{theorem}
\begin{proof} By Theorem~\ref{katorel}, the spectral projection  $P(t)$ is analytic in $\partial_\Gamma.$ Hence by Proposition~\ref{contp}, we have $\rank(P(t)) = \rank(P) = m$ for all $ t \in \partial_\Gamma.$  By Theorem~\ref{spd}, $\sig(A(t)) \cap \mathrm{Int}(\Gamma) \subset \sig_d(A(t))$ and contains $m$ discrete  eigenvalues (counting multiplicity) of $A(t)$ for all $t \in \partial_\Gamma.$ The rest of the proof follows from the  power series of $P(t)$. \end{proof} 

\vone 
It can be shown that $ \alpha_1 = \tr(VP)$ when $ \sig_0 =\{\lam\}$ and $\lam$  is a simple eigenvalue else $ \alpha_1 = \tr(VP)/m,$ see~\cite{kato}.  
The series expansion in Theorem~\ref{main} can also be derived as follows.  It is easy to see that $$AP =  \frac{-1}{2\pi i} \int_{\Gamma} z R(z) dz \text{ and }  \lam_{\mathrm{av}} = \frac{1}{m}\tr(AP) = \frac{1}{m}\tr\left( \frac{-1}{2\pi i} \int_{\Gamma}z R(z) dz \right).$$
Similarly $ \lam_{\mathrm{av}}(t) =  \frac{\tr(A(t)P(t))}{m} = \frac{1}{m}\tr\left( \frac{-1}{2\pi i} \int_{\Gamma}z R(t, z) dz \right)$. Hence by Theorem~\ref{hol}, 
 $$ \lam_{\mathrm{av}}(t) = \frac{1}{m}\tr\left( \frac{-1}{2\pi i} \sum^\infty_{n=0} \left( \int_{\Gamma}z R(t_0, z) [-V R(t_0, z)]^n dz\right) (t-t_0)^n\right).$$ Now taking $ t_0 = 0,$ we have the desired result.

\vone
An asymptotic bound for the arithmetic mean of the eigenvalues  in $\sig_0$ follows immediately from Theorem~\ref{main}. Indeed, we have \be\label{tracebd} |\lam_{\mathrm{av}} - \lam_{\mathrm{av}}(t) | \leq |\alpha_1|\, |t| + \mathcal{O}(|t|^2).\ee As for non-asymptotic bounds, we have the following result, see~\cite{alamjma, alammc, osborn, chatelin}. \von

\begin{theorem} \label{mainbd} Suppose that $ \sig(A)\cap\mathrm{Int}(\Gamma)) = \{\mu\} \subset \sig_d(A)$ and $ \Gamma \subset \rho(A).$  Let $\ell$ and $\nu$ be the algebraic multiplicity and ascent of $\mu,$ respectively.

	Let $\mu_1(t), \ldots, \mu_{\ell}(t)$ be the $\ell$ eigenvalues (counting multiplicities) of $A(t)$ inside $\Gamma $ and
	$\mu_{\mathrm{av}}(t) := (\mu_1(t) + \cdots + \mu_{\ell}(t))/{\ell}$ for $ t \in \partial_{\Gamma}.$ Then the following hold.
	
	(a) There is a $\delta >0$  and a constant $\alpha$ (independent of $t$ and $V$) such that $$ \|t V\| < \delta\Longrightarrow | \mu_j(t) - \mu|^{\nu} \leq  \alpha \|t V\| \ \text{ for } \ j = 1,2, \ldots, \ell.$$
	
	(b) There is a $\delta >0$  and a constant $\beta$ (independent of $t$ and $V$) such that $$ \|t V\| < \delta\Longrightarrow | \mu_{\mathrm{av}}(t) - \mu| \leq  \beta \|t V\|.$$
	
\end{theorem} 

\vone 
\section{Nonlinear eigenvalue problems}  Let $ \Omega \subset \C$ be open and connected. Let  $ T : \Omega \longrightarrow BL(X)$ be holomorphic and regular, that is, $T(z)$ is invertible for some $ z \in \Omega.$ Consider the nonlinear eigenvalue problem  $$ T(\lam) v =0,$$ where $ \lam \in \Omega  $ and  $v \in X$ is nonzero. Nonlinear eigenvalue problems arise in many applications; see~\cite{nlevp, volker1, volker2, MM,  tiss1, WN, alamsafik} and the references therein.  For example,  the nonlinear eigenvalue problem
$$ T(\lambda)v := (\lambda I -A_{0}-\sum\limits _{i=1}^{m}\, A_{i}e^{-\lambda \tau _{i}}) v=0  $$ arises when we seek a solution of the $x(t) :=e^{\lam t}v$  of the delay differential equation
$$
\frac{dx(t)}{dt}= A_{0}x(t)+ \sum\limits _{i=1}^{m}\, A_{i}x(t-\tau _{i}),
$$
where $x(t)\in \, \mathbb{R}^{n}$ is the state variable at time $t$, $A_{i}$'s are $n\times n$ matrices,  and
$0< \, \tau _{1} <  \tau _{2} < \cdots < \tau _{m}$ represent the time-delays~\cite{WN}.

We now show that perturbation results for discrete eigenvalues of bounded linear operators  can be extended to the case of nonlinear eigenvalue problems. To that end, we define the spectrum and the discrete spectrum of $T(z)$  as follows, see~\cite{alamali}. \von

 Let $ T : \Omega \longrightarrow BL(X)$ be holomorphic.  The resolvent set $\rho(T)$ and the spectrum $\sig(T)$ of $T(z)$ are defined by
 $$ \rho(T) := \{ \lam \in \Omega : (T(\lam))^{-1} \in BL(X)\} \text{ and } \sig(T) := \Omega \setminus \rho(T).$$ If $ \rho(T) \neq \emptyset$ then $T(z)$ is said to be {\em regular.} 
 
\von

\begin{definition}   Let $ T : \Omega \longrightarrow BL(X)$ be holomorphic and regular. 
	\begin{itemize}
	\item[(a)] Then $\mu \in \sig(T) $ is said to be an eigenvalue of $T(z)$ if there is a nonzero vector $ v \in X$ such that $T(\mu) v = 0.$ The vector $v$ is called an eigenvector of $T(z)$ corresponding to the eigenvalue $\mu.$
	
	\item[(b)] An eigenvalue  $\mu$ is said to be a discrete eigenvalue of $T(z)$ if $\mu$ is an isolated point of $ \sig(T)$ and  $T(\mu)$ is Fredholm.  The discrete spectrum  $\sig_d(T)$ of $T(z)$ is the set of all discrete eigenvalues of $T(z).$
	
	\item[(c)]  Let $ \mu \in \sig_d(T).$ Then $\nu$ is called the ascent of the eigenvalue  $\mu$ if $\mu$ is a pole of $ T(z)^{-1}$  order $\nu$. 
\end{itemize} 		
\end{definition}

\von

We always assume that $ \Gamma \subset \Omega$ is a positively oriented rectifiable simple closed curve such that $\mathrm{Int}(\Gamma) \subset \Omega.$ It is well known that the number of zeros (counting multiplicity)  of a holomorphic function $ f : \Omega \longrightarrow \C$ inside a curve $\Gamma$ is given by the logarithmic residue $$ \mathrm{n}(\Gamma, f) = \frac{1}{2 \pi i} \int_{\Gamma} f'(z) f(z)^{-1} dz,$$ where $f'(z)$ is the derivative of $f(z).$ An operator analogue of the logarithmic residue theorem is proved in~\cite{ggk, gs} which gives the algebraic multiplicity of a discrete eigenvalue $\mu \in \sig_d(T)$.  If $\Gamma \subset \rho(T)$ and $\sig(T) \cap\mathrm{Int}(\Gamma) = \{\mu\} \subset \sig_d(T)$ then 
$$  m(\mu, T) = \mathrm{Tr}\left(\frac{1}{2 \pi i} \int_{\Gamma} T(z)^{-1} T'(z)  dz\right)$$ is the algebraic multiplicity of $\mu,$ 
where $T'(z)$ is derivative of $T(z)$ with respect to $z.$ 

\vone

Let $ V : \Omega \longrightarrow BL(X)$ be holomorphic. Consider the one parameter family of holomorphic operator-valued function $W(t, z) := T(z) + t V(z)$ for $ t \in \C.$ We now briefly describe how to extend  spectral perturbation theory for discrete eigenvalues of a bounded linear operator to the case of a holomorphic  operator-valued function.  \von

Rouche's theorem states that if $ f, g : \Omega \longrightarrow \C$ are holomorphic and $\Gamma$ is a simple closed curve in $\Omega$ and if $ \max_{z \in \Gamma} | g(z) f(z)^{-1}| < 1$ then $f$ and $f+g$ have the same number   of zeros (counting multiplicity) inside the curve $\Gamma.$ An operator analogue of Rouche's theorem is proved in~\cite{gs, ggk} which states that if $ T(z)$ and $ V(z)$ are Fredholm for all $z \in \Omega$ and if $ \max_{z \in \Gamma} \| V(z) T(z)^{-1}\| < 1$ then $ T(z) + V(z)$ is Fredholm for all $z \in \mathrm{Int}(\Gamma)$  and that  $T(z)$ and $T(z) + V(z)$  have the same number of eigenvalues (counting multiplicity) inside the curve $\Gamma.$  It is shown in~\cite{alamali} that  Rouche's theorem still holds under a weaker assumption, namely,  when  $\| V(z) T(z)^{-1}\|$ is replaced with the spectral radius $r_{\sig}( V(z) T(z)^{-1})$. \von

\begin{theorem}\cite{alamali} \label{eigcount} Let $ T,  V : \Omega \longrightarrow L(X)$ be holomorphic and $T(z)$ be regular. Suppose that $ \sig(T)\cap \mathrm{Int}(\Gamma) \subset \sig_d(T)$ and that $ \max_{z\in \Gamma}r_{\sig}( V(z) T(z)^{-1}) <1$.  Then both $T(z)$ and $T(z) +V(z)$ are Fredholm operators of index zero for  all $z\in \mathrm{Int}(\Gamma)$  and that $T(z)$ and $ T(z) + V(z)$ have the same number  of discrete eigenvalues (counting multiplicity) inside the curve $\Gamma.$ 
\end{theorem}

\vone
The key to extending the spectral perturbation results for a bounded linear operator to the case of a holomorphic operator-valued function is the concept of linearization.  Linearization is a powerful technique that transforms a nonlinear eigenvalue problem to a linear eigenvalue problem. Linearizations of polynomial and rational eigenvalue problems have been studied extensively in the literature; see~\cite{glr,mmmm, AV, blkmin,tdm, rafinami1, rafinami2, dopico} and references therein. For example, the linear eigenvalue problem in (\ref{gep}) is a linearization of the quadratic eigenvalue problem in (\ref{quad}). Linearization of a holomorphic operator-valued function is defined as follows.\vone

 \begin{definition}\cite{alamali}\label{linearization}  Let $X$ and $\mathbb{X}$ be Banach spaces. Let $T : \Omega \longrightarrow BL(X)$ be holomorphic and regular.  Then an operator $\mathbb{T} \in BL(\mathbb{X})$ is said to be a linearization of $T(z)$ on $\Omega$ if there exists a Banach space $Z$ such that $ \mathbb{T}-z  I$ and $ T(z)  \oplus I_Z$ are equivalent on $\Omega$  $$ \mathbb{T}-z  I  \sim_{\Omega}  T(z)  \oplus I_Z,$$
	that is, there exist holomorphic and invertible operator  functions $E : \Omega \to BL(X \oplus Z,  \mathbb{X})$ and $F : \Omega \to BL(\mathbb{X}, X \oplus Z )$ such that
	$ F(z)(zI - \mathbb{T})E(z) = T(z) \oplus I_Z\   \text{for all} \  z \in \Omega, $
	where $I$ is the identity operator on $\mathbb{X}$ and $I_Z$ is the identity operator on $Z.$
\end{definition}

\vone 
Observe that  if $\mathbb{T}$ is a linearization of $T (z)$ on $\Omega$ then   $\sig(T) = \sig(\mathbb{T}) \cap \Omega.$  Also, if  $T(z)$ is Fredholm then so is $ \mathbb{T}- z I.$  Hence if $ \mu$ is a discrete eigenvalue of $ T(z)$ then $ \mu$ is a discrete eigenvalue of $ \mathbb{T}$, see~\cite{alamali}. We have the following result.

\von 
\begin{theorem}\cite{alamali} \label{thm1} Let $T : \Omega \longrightarrow BL(X)$ be  holomorphic and regular. Let $ \mathbb{T} \in BL( \mathbb{X})$ be a linearization of $T(z)$ on $ \Omega.$   Then $\sig(T) = \sig(\mathbb{T})\cap \Omega.$
	Let $ \mu\in \sig(T)$ be an isolated point and  $ \Gamma \subset \rho(T)$ be such that $\sig(T) \cap \mathrm{Int}(\Gamma) = \{\mu\}$ and $\mathrm{Int}(\Gamma) \subset \Omega.$  Then $ \Gamma \subset \rho(\mathbb{T})$. Further,
	$\mu \in \sig_d(T) \iff \mu \in \sig_d(\mathbb{T}).$

	Suppose that $\mu \in \sig_d(T).$   Let $ \mathbb{P} := \frac{1}{2 \pi i} \int_{\Gamma} (z I - \mathbb{T})^{-1} dz$ be the spectral projection associated with $\mathbb{T}$ and $\mu.$ Then  we have $$  m(\mu, T) = \mathrm{Tr}\left(\frac{1}{2 \pi i} \int_{\Gamma} T(z)^{-1}T'(z)  dz\right) = \mathrm{Tr}\left( \frac{1}{2 \pi i} \int_{\Gamma} (z I - \mathbb{T})^{-1} dz \right)= \mathrm{rank}(\mathbb{P}),$$
	that is, the algebraic multiplicity of $\mu$ as an eigenvalue of $T(z)$ is the same as the algebraic multiplicity of $\mu$ as an eigenvalue of $\mathbb{T}.$ 	Further, the ascent of $\mu$ as an eigenvalue of $T(z)$ is the same as the ascent of $\mu$ as an eigenvalue of $\mathbb{T}.$
\end{theorem}

\vone Gohberg-Kaashoek-Lay~\cite{gkl, ggk} constructed  a linearization of $T(z)$, which we refer to as GKL-linearization, as follows.  For the rest of this section, assume that  $\Omega $ is a simply connected bounded open subset of $\C$ such that the boundary $\partial \Omega$ is a simple closed rectifiable curve and is oriented positively.
Let $ C(\partial \Omega, X)$ denote the Banach space of all $X$-valued continuous functions on $\partial\Omega$ endowed with the supremum norm $$C(\partial \Omega, X) := \{ f: \partial \Omega  \longrightarrow X \ | \ f \ \text{ is continuous} \} \mbox{ and } \|f\|_{\infty} := \sup_{z \in \partial \Omega}\|f(z)\|.$$

\begin{theorem}\cite{gkl, ggk} \label{lin} Let $ T : \Omega \longrightarrow BL(X)$ be holomorphic and continuous on the closure $\bar \Omega.$ Suppose that $T(z)$ is regular.   Define $ \mathbb{T} : C(\partial \Omega, X)\longrightarrow C(\partial \Omega, X)$ by
	$$
	(\mathbb{T} f)(z) := z f(z)-\frac{1}{2\pi i} \int_{\partial \Omega}(I-T(w))f(w)dw,$$
	 where $I$ is the identity operator on $X.$   Then $\mathbb{T}$ is a bounded operator and there exists a Banach space $Z$ such that $ T(z) \oplus I_Z \sim_{\Omega} \mathbb{T}-z I,$ where $I$ is the identity operator on $ C(\partial \Omega,  X).$  Hence $\mathbb{T}$ is a linearization of $T(z)$ on $ \Omega.$
\end{theorem}

\von The Banach space $Z$ in Theorem~\ref{lin} is constructed as follows.  Let  $ \lam_0 \in \Omega$ be arbitrary but fixed. Since $ \Omega$ is open,  $ \lam_0 \notin \partial \Omega.$  Define \be \label{z}  Z := \{f\in{C(\partial \Omega, X)} : \; \int_{\partial \Omega} \frac{1}{z-\lam_0}f(z)dz=0 \}.\ee Then $ T(z) \oplus I_Z \sim_{\Omega} \mathbb{T}-z I$ on $\Omega$; see~\cite{alamali, gkl,ggk}. 
The next result describes the bijective correspondence between  eigenvectors of $T(z)$ and  $ \mathbb{T}.$ This correspondence allows us to recover eigenvectors of $T(z)$ from those of $\mathbb{T}.$

\vone
\begin{proposition}\cite{alamali}\label{reco} Let $ \mu \in  \sig(T)$ be an eigenvalue. Let $ \mathbb{T}$ be be the GKL-linearization of $T(z)$.
	Then the linear maps $ \mathcal{E} : N(T(\mu)) \longrightarrow N(\mathbb{T} -\mu I), x \longmapsto \mathcal{E}x,$ and $ \mathcal{F} : N(\mathbb{T}- \mu I) \longrightarrow N(T(\mu))$ given by
	$$ (\mathcal{E}x)(w) := \frac{x}{ \mu - w}, \ w \in \partial \Omega, \mbox{ and } \mathcal{F}f := \frac{1}{2\pi i} \int_{\partial \Omega} \frac{(\mu - w)}{w-\lam_0} f(w)dw$$ are isomorphisms, where $\lam_0 \in \Omega$ is fixed as in (\ref{z}).
\end{proposition}

\vone
Note that $\bar \Omega$ denotes the topological closure of $\Omega.$ Consider the function space 
  $$ \mathbb{H}(\Omega, BL(X)) := \{ T : \Omega \longrightarrow BL(X) \ | \ T \text{ is holomorphic on } \Omega  \ \text{ and continuous on } \bar \Omega\}.$$
For $ V \in \mathbb{H}(\Omega, BL(X)),$ define $ \|V\|_{\partial\Omega} := \frac{1}{2 \pi } \int_{\partial\Omega} \|V(z)\| \, |dz|.$
Since $ z \longmapsto \|V(z)\|$ is subharmonic on $\Omega,$ it follows that $ \|V\|_{\partial\Omega} = 0 \Longrightarrow V(z) = 0$ for all $ z \in \Omega \cup\partial\Omega.$ In fact, $ \|\cdot\|_{\partial \Omega}$ defines a norm on $ \mathbb{H}(\Omega, BL(X)).$  Next, define $ \mathbb{V} : C(\partial\Omega, X) \longrightarrow C(\partial\Omega, X)$ by  $$ \mathbb{V}f := \frac{1}{2 \pi i} \int_{\Gamma} \mathcal{I} V(z)f(z) dz,$$ where $\mathcal{I}$ is the imbedding of $X$ into $C(\partial\Omega, X).$  Then  $ \|\mathbb{V}f\|_{\infty} \leq \|V\|_{\partial\Omega} \|f\|_{\infty}.$  

\vone 

 Consider $ W(t, z) := T(z) + t V(z)$ for $ t \in \C.$ Then  the GKL-linearization $\mathbb{T}(t)$ of $W(t, z)$  can be written as $ \mathbb{T}(t) = \mathbb{T}+ t \mathbb{V}$ for $ t \in \C,$ where $\mathbb{T}$ is the GKL-linearization of $T(z).$
Let $ \mu$ be a discrete eigenvalue of $T(z)$ of algebraic multiplicity $\ell.$ Let $ \Gamma \subset \rho(T)$ be such that $\sig(T)\cap \mathrm{Int}(\Gamma) = \{\mu\}.$ Let $$ \partial_{\Gamma} :=\{ t \in \C : \max_{z\in \Gamma} r_{\sig}\left( V(z) T(z)^{-1}\right) < 1\}.$$   Then it can be shown \cite[Theorem~5.2]{alamali} that  $ \Gamma \subset \rho(\mathbb{T}(t))$ for all $ t \in \partial_{\Gamma}$  and the spectral projection
$$ \mathbb{P}(t):= \frac{1}{2\pi i}\int_{\Gamma}(zI- \mathbb{T}(t))^{-1}dz $$ is holomorphic on $\partial_{\Gamma}.$ Hence  $\mathrm{rank}(\mathbb{P}(t)) = \mathrm{rank}(\mathbb{P}) =\ell$ for all $ t \in \partial_{\Gamma},$ where $\mathbb{P}  $ is the spectral projection associated with $\mathbb{T} $ and $\mu.$ By Theorem~\ref{eigcount}, $W(t,z)$ has $\ell$ discrete eigenvalues (counting multiplicity)  $ \mu_1(t), \ldots, \mu_{\ell}(t)$ inside the curve $\Gamma.$ Set $$ \mu_{\mathrm{av}}(t) := (\mu_1(t) + \cdots + \mu_{\ell}(t))/{\ell}.$$ Then by (\ref{tracebd})  (see also \cite[p.~405]{kato}), we have  $ \mu_{\mathrm{av}}(t) = \mu + \frac{1}{\ell} \mathrm{Tr}( \mathbb{VP}) t + \mathcal{O}(|t|^2)$, where $\mathbb{P}$ is the spectral projection associated with $\mathbb{T}$ and $ \lam.$  Hence we have the first order bound
$$ |\mu_{\mathrm{av}}(t) - \mu | \leq  \frac{1}{\ell} |\mathrm{Tr}(\mathbb{VP}) | \,|t| + \mathcal{O}(|t|^2).$$

The one parameter family of operators $ \mathbb{T}(t) = \mathbb{T}+ t \mathbb{V}, t \in \C,$ which is a linearization of the one parameter operator-valued function $W(t, z) := T(z) +t V(z),\; t \in \C,$ can now be utilized to derive various perturbation  bounds for discrete eigenvalues of $W(t, z).$ Indeed, we have the following result, see~\cite{alamali}.

\vone 
\begin{theorem} Let $ \mu $ be a discrete eigenvalue of $T(z)$ of  algebraic multiplicity $\ell$ and ascent $\nu$  such that $\sig(T)\cap \mathrm{Int}(\Gamma) = \{\mu\}.$ 	
	Let $\mu_1(t), \ldots, \mu_{\ell}(t)$ be the $\ell$ discrete eigenvalues (counting multiplicity) of $W(t,z)$ inside the curve $\Gamma$ for $ t \in \partial_\Gamma.$ Set
	$$\mu_{\mathrm{av}}(t) := (\mu_1(t) + \cdots + \mu_{\ell}(t))/{\ell}$$ for $ t \in \partial_{\Gamma}.$ Then the following hold.
	
	(a) There is a $\delta >0$  and a constant $\alpha$ (independent of $t$ and $V$) such that $$ \|t V\|_{\partial\Omega} < \delta\Longrightarrow | \mu_j(t) - \mu|^{\nu} \leq  \alpha \|t V\|_{\partial\Omega} \ \text{ for } \ j = 1,2, \ldots, \ell.$$
	
	(b) There is a $\delta >0$  and a constant $\beta$ (independent of $t$ and $V$) such that $$ \|t V\|_{\partial\Omega} < \delta\Longrightarrow | \mu_{\mathrm{av}}(t) - \mu| \leq  \beta \|t V\|_{\partial\Omega}.$$
	
\end{theorem}

\begin{proof} Consider the operator $ \mathbb{T}(t) := \mathbb{T}+ t \mathbb{V}$ for $ t \in \C. $ Then by Theorem~\ref{mainbd}  there is a $\delta >0$ and a constant $\alpha$ independent of $t$ and $\mathbb{V}$ such that $$\|t \mathbb{V}\|_{\infty} < \delta\Longrightarrow | \mu_j(t) - \mu|^{\nu} \leq  \alpha \|t \mathbb{V}\|_{\infty} \ \text{ for } \ j = 1: \ell.$$
	Similarly, there is a $\delta >0$ and a constant $\beta$ independent of $t$ and $\mathbb{V}$ such that $$\|t \mathbb{V}\|_{\infty} < \delta\Longrightarrow | \mu_{\mathrm{av}}(t) - \mu| \leq  \beta \|t \mathbb{V}\|_{\infty}.$$ Now the desired results follow from the fact that $ \| \mathbb{V}\|_{\infty} \leq \| V\|_{\partial\Omega}.$ \end{proof}

\vone

\end{document}